\documentclass[11pt]{amsart}

\makeatletter{}\usepackage{amssymb}
\usepackage{amsmath}

\usepackage{color}
\usepackage{graphicx}
\usepackage[foot]{amsaddr}
\usepackage{enumitem}
\usepackage{caption}
\usepackage{hyperref}
\usepackage[mathscr]{eucal}
\usepackage{multicol}

\DeclareCaptionLabelFormat{rightline}{\rightline{\bothIfFirst{#1}{ }#2}}
\usepackage[font=small,floatrowsep=none,captionskip=5pt]{floatrow}

\providecommand{\prob}[1]{{\ensuremath{\mathrm{P}}\mspace{-2mu}\left[#1\right]}}\providecommand{\E}[1]{{\ensuremath{\mathrm{E}}\mspace{-2mu}\left[#1\right]}}\providecommand{\var}[1]{{\ensuremath{\mathrm{Var}}\mspace{-2mu}\left[#1\right]}}\providecommand{\tol}{\mathrm{TOL}}
\providecommand{\rset}{\mathbb{R}}
\providecommand{\rsetp}{\mathbb{R}_{+}}

\providecommand{\nset}{\mathbb{N}}

\providecommand{\Order}[1]{ {\ensuremath{ \mathcal O\left( #1 \right)}} }
\providecommand{\order}[1]{ {\ensuremath{ o\left( #1 \right)}} }

\providecommand{\mykeywords}[1]{\par \textbf{Keywords:} #1}
\providecommand{\subclass}[1]{\par{ \bf Class:} #1}
\providecommand{\CWork}{\ensuremath{C_{\text{\textnormal{work}}}}}
\renewcommand{\vec}[1]{{\ensuremath{{\boldsymbol #1}}}}
\providecommand{\valpha}{{{\ensuremath{\vec{\alpha}}}}}

\providecommand{\vdelta}{{\ensuremath{\vec{\delta}}}}
\providecommand{\vtau}{{\ensuremath{\vec{\tau}}}}
\providecommand{\intlim}[2]{{\ensuremath{ \int_{ \{ #1 \::\: #2 \}} }}}
\providecommand{\sumlim}[2]{{\ensuremath{ \sum_{ \{ #1 \::\: #2 \}}
    }}}

\providecommand{\eqdef}{{\ensuremath{ := } }}

\providecommand{\indset}{{\ensuremath{ I }}}
\providecommand{\indsetA}{{\ensuremath{ I_1 }}}
\providecommand{\indsetB}{{\ensuremath{ I_2 }}}
\providecommand{\indsetC}{{\ensuremath{ I_3 }}}
\providecommand{\indsetBC}{{\ensuremath{ \hat I }}}
\providecommand{\estS}{{\ensuremath{ \Delta \mathscr{S} }}}

\providecommand{\workEst}{{\ensuremath{ \widetilde W }}}
\providecommand{\workRem}{{\ensuremath{ \widetilde W_1 }}}

\providecommand{\inchangesA}[1]{#1}

\newcommand{\constRecall}[2]{#1 & in \eqref{#2} on page
  \pageref{#2}. \\}

\allowdisplaybreaks[1]
\title{Multi-Index {M}onte {C}arlo: When Sparsity Meets Sampling}

\author{
  Abdul--Lateef~Haji--Ali
  \and
  Fabio~Nobile
  \and
  Ra\'{u}l~Tempone
}

\address[abdullateef.hajiali@kaust.edu.sa, raul.tempone@kaust.edu.sa]
{Applied Mathematics and Computational Sciences, King Abdullah University of Science and Technology (KAUST), Thuwal, Saudi Arabia.}

\address[fabio.nobile@epfl.ch]{MATHICSE-CSQI,   {Ecole Polytechnique F\'ed\'erale de Lausanne}, Switzerland.}

\theoremstyle{plain}
\newtheorem{theorem}{Theorem}[section]
\newtheorem{lemma}{Lemma}[section]
\theoremstyle{remark}
\newtheorem{remark}{Remark}[section]
\newtheorem*{remark*}{Remark}
\theoremstyle{definition}
\newtheorem{example}{Example}
\newtheorem*{example*}{Example}
\graphicspath{{src/imgs/}}

\begin{document}
\makeatletter{}\makeatletter{}\begin{abstract}
  We propose and analyze a novel Multi-Index Monte Carlo (MIMC) method
  for weak approximation of stochastic models that are described in
  terms of differential equations either driven by random measures or
  with random coefficients. The MIMC method is both a stochastic
  version of the combination technique introduced by Zenger, Griebel
  and collaborators and an extension of the Multilevel Monte Carlo
  (MLMC) method first described by Heinrich and Giles. Inspired by
  Giles's seminal work, we use in MIMC high-order mixed differences
  instead of using first-order differences as in MLMC to reduce the
  variance of the hierarchical differences dramatically. This in turn
  yields new and improved complexity results, which are natural
  generalizations of Giles's MLMC analysis and which increase the
  domain of the problem parameters for which we achieve the optimal
  convergence, $\mathcal{O}(\tol^{-2}).$ \inchangesA{Moreover, in
    MIMC, the rate of increase of required memory with respect to
    $\tol$ is independent of the number of directions up to a
    logarithmic term which allows far more accurate solutions to be
    calculated for higher dimensions than what is possible when using
    MLMC}.

  We motivate the setting of MIMC by first focusing on a simple full
  tensor index set. We then propose a systematic construction of
  optimal sets of indices for MIMC based on properly defined profits
  that in turn depend on the average cost per sample and the
  corresponding weak error and variance. Under standard assumptions on
  the convergence rates of the weak error, variance and work per
  sample, the optimal index set turns out to be the total degree (TD)
  type. In some cases, using optimal index sets, MIMC achieves a
  better rate for the computational complexity than the corresponding
  rate when using full tensor index sets. We also show the asymptotic
  normality of the statistical error in the resulting MIMC estimator
  and justify in this way our error estimate, which allows
  both the required accuracy and the confidence level in our
  computational results to be prescribed. Finally, we include
  numerical experiments involving a partial differential equation
  posed in three spatial dimensions and with random coefficients to
  substantiate the analysis and illustrate the corresponding
  computational savings of MIMC.

  \mykeywords{Multilevel Monte Carlo, Monte Carlo, Partial
    Differential Equations with random data, Stochastic Differential
    Equations, Weak Approximation, Sparse Approximation, Combination
    technique} \subclass{65C05 \and 65N30 \and 65N22}
\end{abstract}

\maketitle

\pagestyle{myheadings}
\thispagestyle{plain}
\markboth{Multi-Index {M}onte {C}arlo}{Multi-Index {M}onte {C}arlo}
\makeatletter{}\section{Introduction}\label{s:intro}
The main concept of Multilevel Monte Carlo (MLMC) Sampling was
first introduced for applications in parametric integration by
Heinrich~\cite{heinrich98,hs99}.
Later, for weak approximation of  Stochastic Differential Equations (SDEs)
in mathematical finance, Kebaier~\cite{kebaier05} used a two-level
Monte Carlo technique, effectively using a coarse numerical approximation
as a control variate of a fine one, thus reducing the variance and the required number of samples on the fine grid.
In a seminal work, Giles \cite{giles08} extended this idea to multiple levels and gave it its familiar name: Multilevel Monte Carlo.
Giles introduced a hierarchy of discretizations
with geometrically decreasing grid sizes and optimized the
number of samples on each level of the hierarchy.
This resulted in a reduction in the computational burden
 from $\Order{\tol^{-3}}$ of the standard Euler-Maruyama Monte Carlo method with accuracy $\tol$ to
$\Order{\log{(\tol)}^2\tol^{-2}}$, assuming that the work to generate
a single realization on the finest level is $\Order{\tol^{-1}}$.
More recently, \cite{gs13} reduced this computational complexity to $\Order{\tol^{-2}}$
by using antithetic control variates with MLMC in multi-dimensional SDEs with smooth and piecewise smooth payoffs.
The MLMC method has also been extended and applied to a wide variety of
applications, including jump diffusions \cite{xg12} and Partial Differential Equations (PDEs) with random coefficients
\cite{bsz11,cst13,cgst11,gr12,tsgu13,haji_CMLMC,haji_opt}.
The goal in these applications is to compute a scalar quantity of
interest that is a functional of the solution of a PDE with random coefficients.
In \cite[Theorem 2.3]{tsgu13}, it has been proved that there is an optimal complexity rate similar to the previously mentioned one,
but this rate depends on the dimensionality of the problem, the relation between
the rate of variance convergence of the discretization method of the PDE and
the work complexity associated with generating a single sample of the quantity of interest.
In fact, in certain cases, the computational complexity can achieve the  optimal rate, namely $\Order{\tol^{-2}}$.

More recently, sparse approximation techniques \cite{bungartz2004sparse} have been coupled with MLMC in other works.
In \cite{mishra2012sparse}, the MLMC sampler was combined with a sparse tensor approximation method to
estimate high-order moments of the finite volume approximate solution of a hyperbolic conservation law that has random initial data.
Moreover, in \cite{Harbrecht_H13_MLSC, van2014multilevel},  new techniques were developed using sparse-grid stochastic collocation methods
instead of Monte Carlo sampling in a multilevel setting that resembles that of MLMC.

In the present work, we follow a different approach by introducing a
stochastic version of a sparse combination technique
\cite{zenger1990sparse,Griebel_first_combination92,Convergence_combination_94,Bungartz_Laplace_combination96,bungartz2004sparse,Hegland_Angle_Combination07}
in the construction of a new {\em Monte Carlo} sampler, which we refer
to as Multi-Index Monte Carlo (MIMC). MIMC can be seen as a
generalization of the standard Multilevel Monte Carlo Sampling method.
This generalization departs from the notion of one-dimensional levels
and first-order differences and instead uses multidimensional levels
and high-order mixed differences to reduce the variance of the
resulting estimator and its corresponding computational work
drastically. The goal of MIMC is to achieve the optimal complexity
rate of the Monte Carlo sampler, $\Order{\tol^{-2}}$, in a larger
class of problems and to provide better convergence rates in other
classes.  The main results of our work are summarized in
Theorems~\ref{thm:gmlmc_ft} and \ref{thm:gmlmc_td_opt}. These theorems
contain the optimal work estimates of MIMC when using full tensor
index sets and total degree index sets, respectively. {The results of
  MIMC with full tensor index sets are meant to motivate the setting
  of MIMC in a simple framework. However, we later show in this work
  that the total degree index sets are optimal given certain
  assumptions. In fact, we show that the rate of computational
  complexity of MIMC when using optimal index sets, and the
  corresponding conditions on the rate of weak convergence, are
  independent of the dimensionality of the underlying problem.}

In the next section, we start by motivating the class of problems we
consider and we introduce some notation that is used throughout this
work. Section~\ref{sec:gmlmc} introduces MIMC and lists the necessary
assumptions. Section~\ref{ss:full_tensor} presents the computational
complexity of a full tensor index set, and Section~\ref{ss:td_set}
motivates an optimal total degree index set and shows the
computational complexity of MIMC when using this index set. Next,
Section~\ref{s:res} presents the numerical experiments to substantiate
the derived results. Section~\ref{s:conc} summarizes the work and
outlines future work. \inchangesA{Finally, the Appendix contains
  proofs of different lemmas used in this paper including a proof of
  the asymptotic normality of the MIMC estimator. Moreover, Appendix
  \ref{app:defs} contains, for convenience, definitions of important
  quantities that are used throughout this paper.  }
 
\makeatletter{}\subsection{Problem Setting}\label{sec:hier_intro}
Let $S=\Psi(u)$ denote a real-valued functional applied to the unique solution, $u$, of an underlying stochastic model.
We assume that $\Psi$ is a smooth functional with respect to $u$.
Here, smoothness is characterized by $S$ satisfying {\bf Assumptions
  1-2} as presented in the next section.
Our goal is to
approximate the expected value of $S$, $\E{S}$, to a given accuracy $\tol$
and a given confidence level. We assume that individual outcomes
of the
underlying solution, $u$, and the evaluation of the functional, $S$, are approximated
by a discretization-based numerical scheme characterized by a
multidimensional discretization parameter, $\vec h$.
{For instance, for a multidimensional PDE, the vector $\vec h$ could
  represent the space discretization parameter in each direction
  separately, while for a time dependent PDE, the vector $\vec h$ could collect the space and time discretization parameters.}
The value of  the vector, $\vec h$, will govern the weak error and
variance of the approximation of $S$ as we will see below. To motivate this setting,
we now give one example and identify  the corresponding numerical discretizations,
the discretization parameter, $\vec h$, and the corresponding rates of approximation.
\begin{example}
  \label{ex:spde_problem}
Let $(\Omega,\mathcal{F},P)$ be a complete probability space and
$\mathcal{D} = \prod_{i=1}^d (0, D_i)$ for $D_i \in \rsetp$ be a hypercube
domain in
$\rset^d$.
 The solution $u: \mathcal{D} \times \Omega \to \rset$ here solves almost surely (a.s.)
the following equation:
\begin{equation}\label{eq:standard_PDE}
\begin{aligned}
    -\nabla \cdot \left( a(\vec x; \omega) \nabla u(\vec x; \omega) \right) &= f(\vec x; \omega) &\text{ for }
            \vec x  \in \mathcal D,  \\
     u(\vec x; \omega) &=  0 &\text{ for } \vec x \in \partial \mathcal D.
\end{aligned}
\end{equation}
This example is common in engineering applications like heat
conduction and groundwater flow. Here, the value of the diffusion
coefficient and the forcing are represented by random fields, yielding
a random solution and a functional to be approximated in the mean.
Given certain assumptions on coercivity and continuity
related to the random coefficients $a$ and $f$ \cite{tsgu13}, the
solution to \eqref{eq:standard_PDE} exists and is unique. Actually,
$u$ depends continuously on the coefficients of
\eqref{eq:standard_PDE}. A standard approach to approximate the
solution to \eqref{eq:standard_PDE} is to use Finite Elements on
Cartesian meshes. In such a setting, the vector parameter $\vec h =
(h_1,\ldots,h_d)>0$ contains
the mesh sizes in the different canonical directions and the corresponding
approximate solution is denoted by $u_h(\omega)$. Let
$r:\mathcal{D}\to\rset$ be a smooth function and let $\Psi(u) =
\int_{\mathcal{D}} u(x) r(x) dx$ be a linear functional. Our goal here
is to approximate $\E{\int_{\mathcal{D}} u(x) r(x) dx}.$

\end{example}
To particularize our set of discretizations, let us now introduce
integer multi indices, $\valpha \in \nset^d.$
Throughout this work, we  use discretization vectors of the form
\begin{equation*}h_i = h_{i,0} \beta_i^{-\alpha_i}\text{  with given constants $h_{0,i} > 0$ and $\beta_i > 1$ for $i=1,\ldots,d$.}
\end{equation*}
Correspondingly, we index our discrete approximations to $S$ by $\valpha$, denoting them as $\{S_\valpha\}_{\valpha \in \nset^d}$.
In addition, we make the standard assumption that
 $\E{S_\valpha} \to \E{S}$ as $\min_{1\leq i\leq d}{\alpha_i}
 \to \infty$. Finally, for later use, we define $|\valpha| = \sum_{i=1}^d \alpha_i$.

\makeatletter{}\section{Multi-Index Monte Carlo}\label{sec:gmlmc}
Here we introduce the MIMC discretization.  To this end, we begin by defining a first-order difference operator along direction $1\le i\le d$, denoted by $\Delta_{i}$, as follows:
\[ \Delta_{i} S_\valpha = \begin{cases}
  S_{\valpha} - S_{\valpha-\vec e_i}, & \text{if }\alpha_i > 0,\\
  S_\valpha                  & \text{if } \alpha_i=0,
  \end{cases}\]
  with $\vec e_i$ being the canonical vectors in $\rset^d$, i.e.
  $(\vec e_i)_j = 1$ if $j=i$ and zero otherwise.
For later use, we also define recursively the first-order mixed difference operator,
$ \Delta =  \otimes_{i=1}^d \Delta_i =  \Delta_1 (\otimes_{i=2}^d
\Delta_i) = \Delta_d (\otimes_{i=1}^{d-1} \Delta_i) .$

\begin{example*}[$d=2$]
  In this case, letting $\valpha = (\alpha_1,\alpha_2)$, we have
  \begin{align*}
    \Delta S_{(\alpha_1,\alpha_2)} &= \Delta_2 (\Delta_1 S_{(\alpha_1,\alpha_2)})\\
    &= \Delta_2 \left( S_{\alpha_1,\alpha_2} - S_{\alpha_1-1,\alpha_2} \right) \\
    &= \left( S_{\alpha_1,\alpha_2} - S_{\alpha_1-1,\alpha_2} \right) -  \left( S_{\alpha_1,\alpha_2-1} - S_{\alpha_1-1,\alpha_2-1} \right).
  \end{align*}

  Notice that in general, $\Delta S_\valpha$ requires $2^d$ evaluations of $S$
  at different discretization parameters, the largest work
  of which corresponds precisely to the index appearing in $\Delta S_{{\valpha}}$, namely $\valpha = {(\alpha_1,\alpha_2)}$.
\end{example*}
Let  $\estS_\valpha$ be an {\em unbiased} estimator of
$\Delta S_\valpha$.
In the trivial case, $\estS_\valpha = \Delta S_\valpha$ for
all $\valpha \in \nset^d$. However,
$\estS_\valpha$ can be taken to be more complicated such that it has a smaller
variance than that of $\Delta S_\valpha$, for example by constructing an antithetic
estimator similar to \cite{gs13}.
In any case, the MIMC estimator can be written as:
\begin{align} \label{eq:gmlmc} \mathcal{A} = \sum_{\valpha \in
    \mathcal{I}} \frac{1}{M_\valpha} \sum_{m=1}^{M_{\valpha}} \estS_\valpha(\omega_{\valpha,m}),
\end{align}
where $\mathcal{I}{\subset\nset^d}$ is an index set and $M_\valpha$ is
an integer number of samples for each $\valpha \in \mathcal{I}$. Here,
$\omega_{\valpha,m}$ are  independent, identically distributed
(i.i.d.) realizations of the underlying random inputs, $\omega$.
Denote $\var{\estS_\valpha} = V_\valpha$ and $\left|\E{\estS_\valpha}\right| = |\E{\Delta
   S_\valpha}| =
E_\valpha$. Moreover, denote by $W_\valpha$ the average work required to
compute a realization of $\estS_\valpha$. Then, the expected value of the total work corresponding to the estimator, $\mathcal{A}$, is
\begin{equation}\label{eq:total_work}\text{Total work} = W = \sum_{\valpha \in \mathcal{I}} W_\valpha M_\valpha. \end{equation}
Moreover, by independence, the total variance of the estimator is
\begin{equation*}\var{\mathcal{A}} =  \sum_{\valpha \in \mathcal{I}} \frac{V_\valpha}{M_\valpha}.
\end{equation*}

The objective of the MIMC estimator, $\mathcal{A}$, is to achieve a certain accuracy constraint of the form\begin{equation}\label{eq:accuracy_constr}
 P(
 |\mathcal{A} - \E{S}| \leq \tol
 )\
  \ge 1-\epsilon
\end{equation}
for a given accuracy  $\tol$ and a given confidence level determined
by $0<\epsilon \ll 1$.
Here, we further split the accuracy budget between
the bias and statistical errors,  imposing the following, more restrictive, two constraints instead:
\begin{align}
  \label{eq:bias_const} \text{Bias constraint:}  & &|\E{\mathcal{A}-S}| \leq (1-\theta) \tol, \\
  \label{eq:stat_const0} \text{Statistical constraint:}& & P\left( |\mathcal{A} - \E{\mathcal{A}}| \leq \theta\tol \right)
 \ge 1-\epsilon.
 \end{align}
 Throughout this work, the value of the splitting parameter, $\theta \in (0,1)$, is assumed to be given and remains fixed;  satisfying \eqref{eq:bias_const}
 and \eqref{eq:stat_const0} thus implies that
 \eqref{eq:accuracy_constr} is satisfied. {We refer to \cite{haji_CMLMC,haji_opt} for an analysis of the role of $\theta$ on standard MLMC simulations.}
Motivated by the asymptotic normality of the estimator, $\mathcal{A}$,  shown in Appendix \ref{app:clt}, we replace \eqref{eq:stat_const0}
by
\begin{equation}\label{eq:stat_const}
\var{\mathcal{A}} =  \sum_{\valpha \in \mathcal{I}} \frac{V_\valpha}{M_\valpha} \leq \left( \frac{\theta \tol}{C_\epsilon} \right)^2.
 \end{equation}
 Here, $0<C_\epsilon$ is such that $\Phi(C_\epsilon) = 1 -
 \frac{\epsilon}{2}$, where $\Phi$ is the cumulative distribution
 function of a standard normal random variable.  Using the following
 notation,
\begin{equation}     \label{eq:tolS} \tol_S = \frac{\theta \tol}{C_\epsilon}, \end{equation}
and optimizing the total work \eqref{eq:total_work} with respect to $M_\valpha \in \rsetp$ subject to the statistical constraint \eqref{eq:stat_const} yields
\begin{equation} \label{eq:optimal_M}
M_\valpha = \tol_S^{-2} \left(\sum_{\vec{\tau} \in \mathcal{I}}
  \sqrt{V_{\vec{\tau}} W_{\vec{\tau}} } \right) \sqrt{\frac{V_\valpha}{W_\valpha}},
\text{ for all }\valpha \in \mathcal{I}.
\end{equation}
Of course, in numerical computations, we usually have to take the
integer ceiling of $M_\valpha$ in expression \eqref{eq:optimal_M} or
perform some kind of integer optimization to find $M_\valpha \in
\nset$ for all $\valpha$, cf. \cite{haji_opt}.
For this reason, and to guarantee that at least one sample is used in each multi-index,
$\valpha$, we assume the bound
\begin{equation} \label{eq:optimal_M_bound}
M_\valpha \leq 1 + \tol_S^{-2} \left(\sum_{\vec{\tau} \in \mathcal{I}}
  \sqrt{V_{\vec{\tau}} W_{\vec{\tau}} } \right) \sqrt{\frac{V_\valpha}{W_\valpha}},
\text{ for all }\valpha \in \mathcal{I},
\end{equation}
and bound the total work as follows:
\begin{equation*}
    W \leq \tol_S^{-2} \left(\sum_{\valpha \in \mathcal{I}} \sqrt{V_\valpha
      W_\valpha} \right)^2 + \sum_{\valpha \in \mathcal{I}} W_\valpha.
\end{equation*}

In the current work, we assume the following
\begin{itemize}
\item \textbf{Assumption 1}: The absolute value of the expected value of $\estS_\valpha$, denoted by $E_\valpha$, satisfies
  \begin{align}\label{eq:assu_1}
    E_\valpha = \left|\E{\estS_\valpha}\right| \le Q_W  \prod_{i=1}^d  \beta_i^{-\alpha_i w_i}
  \end{align}
  for constants $Q_W$  and $w_i > 0$ for $i = 1 \ldots d$.
\item \textbf{Assumption 2}: The variance of $\estS_\valpha$, denoted by $V_\valpha$, satisfies
\begin{align}\label{eq:assu_2}
  V_\valpha = \var{\estS_\valpha} &\le  Q_S   \prod_{i=1}^d \beta_i^{-\alpha_i s_i},
\end{align}
for constants $Q_S$ and ${0<s_i\le 2w_i}$ for $i = 1 \ldots d$.
\item \textbf{Assumption 3}: The average work required to compute a
  realization of $\estS_\valpha$, denoted by $W_\valpha$, satisfies
  \begin{align}
    \label{eq:wl_model}
    W_\valpha \le \CWork \prod_{i=1}^d   \beta_{i}^{\alpha_i \gamma_i} ,
  \end{align}
  for constants $\CWork$ and $\gamma_i>0$ for $i = 1 \ldots d$.
\end{itemize}

\begin{remark}[On  Assumptions 1, 2 and 3]
  \label{rem:assumptions}
With sufficient coefficient regularity, Assumptions 1 and 2 hold for the random linear elliptic PDE in Example 1
when discretized by piecewise multilinear continuous finite elements.
Indeed, there is extensive work on this problem based on mixed regularity analysis by
several authors who have developed combination techniques through the years. Here, we refer to the works \cite{Pflaum_composition_97,Pflaum_composition_99,Griebel_H13_composition} and the references therein.
In Example 1, it is enough to apply such estimates point wise in $\omega$ and then to observe that they can
be integrated in $\Omega$, yielding the desired moment estimates in \eqref{eq:assu_1} and \eqref{eq:assu_2}. In Section \ref{sec:prob-overview}, the numerical example has
isotropic behavior over $d=3$ dimensions, the work exponent appearing
in Assumption 3 satisfies $\gamma \in [1,2],$ and the
error exponents are $w_i=s_i/2 =2$ for $i=1,\ldots,3$,
respectively. These exponents have also been confirmed by numerical
experiments, cf. Figures~\ref{fig:time_vs_dof}, \ref{fig:mlmc_el_vl}
and \ref{fig:mimc_el_vl}.
\end{remark}

Under \textbf{Assumptions 2-3}, we estimate the total work, $W$, by
\begin{equation}
  \aligned
  \label{eq:gmlmc_work_model_detailed}
  W(\mathcal I) \le& \tol_S^{-2} Q_S \CWork \,\left(\sum_{\valpha \in \mathcal{I}}
    \prod_{i=1}^d \exp\left(  \frac{\alpha_i \log(\beta_i)( \gamma_i -
        s_i)}{2}\right) \right)^2  \\ &\qquad + \CWork \sum_{\valpha \in \mathcal I}
  \prod_{i=1}^d \exp(\alpha_i \log(\beta_i) \gamma_i).
  \endaligned
\end{equation}
Notice that the second term of the total work is the work needed to
calculate exactly one sample per each multi-index, $\valpha \in
\mathcal I$. This is the minimum cost of a Monte Carlo estimator and
we need to make sure that it does not dominate the first term of the
bound in \eqref{eq:gmlmc_work_model_detailed}.
We define  $\overline{\vec{g}} \in \rset^d$ with entries $\overline g_i =
\frac{\log(\beta_i)( \gamma_i - s_i)}{2} $, for $ i \in \{1,2,\ldots, d\}$ and define
\begin{align}
  \label{eq:gmlmc_worktilde_model}
 \workEst(\mathcal I) &= \sum_{\valpha \in   \mathcal{I}}
 \exp\left( \overline{ \vec{g}} \cdot \valpha \right),\\
  \label{eq:gmlmc_work1_model}  \workRem(\mathcal I) &= \sum_{\valpha \in \mathcal I}
  \prod_{i=1}^d \exp(\alpha_i \gamma_i \log(\beta_i)),
\end{align}
so that the total work can be written as
\begin{equation}
  \label{eq:gmlmc_work_model}
  W(\mathcal I) \le \tol^{-2} Q_S \CWork \left( \workEst(\mathcal
    I) \right)^2 + \CWork \workRem(\mathcal I).
\end{equation}
Then, assuming for that moment that the first term of the bound is dominating the
second term, $\workRem(\mathcal I)$, we can focus on estimating $\workEst(\mathcal I)$ instead of the total
work, $W(\mathcal I)$. In the theorems below, we state sufficient
conditions to ensure that this assumption is indeed satisfied.

One of our goals in this work
is to motivate a choice for the set of multi indices, $\mathcal{I}=\mathcal{I}(\tol)$, to minimize
$\workEst(\mathcal I)$, as an approximation to the minimization of the total work, $W(\mathcal
I)$, subject to the following constraint:
\begin{equation}
  \label{eq:bias_I}
  \text{Bias}(\mathcal I) = \left| \sum_{\valpha \notin \mathcal{I}}
    \E{\estS_\valpha} \right|\leq \sum_{\valpha \notin \mathcal{I}} E_\valpha \leq (1-\theta) \tol.
\end{equation}
Due to \textbf{Assumption 1}, we can rewrite \eqref{eq:bias_I} as
\begin{equation}
  \label{eq:gmlmc_biastilde_model}
  \widetilde B(\mathcal I) = \sum_{\valpha \notin \mathcal I}
  \prod_{i=1}^d \exp(-\log(\beta_i) w_i \alpha_i) \leq \frac{(1-\theta)\tol}{Q_W}.
\end{equation}
  Moreover, we introduce the following notation for the right-hand
  side in  \eqref{eq:gmlmc_biastilde_model}:
  \begin{equation}
    \label{eq:tolB}
\tol_B = \frac{(1-\theta)\tol}{Q_W}.
  \end{equation}

For later use, we introduce the notation $\indset = \{1,2,\ldots, d\}$ and
we define the following sets of direction indices:
\begin{equation}
  \label{eq:ind_sets}
\aligned
\indsetA &= \{i \in I: s_i>\gamma_i\}, \\
\indsetB &= \{i \in I: s_i=\gamma_i\}, \\
\indsetC &= \{i \in I: s_i<\gamma_i\}, \\
\indsetBC  &= \indsetB \cup \indsetC =\{i \in I:
s_i\leq\gamma_i\}
\endaligned
\end{equation}
to distinguish between directions based on the speed of variance
convergence in a direction compared with the rate of increase in the
computational complexity in that direction. Correspondingly, denote
\begin{equation}
  \label{eq:ind_sets_sizes}
  d_1 = \#\indsetA,\quad d_2 = \#\indsetB, \quad d_3 = \#\indsetC
  \quad\text{and}\quad \hat d = \#\indsetBC.
\end{equation}

\subsection{Full Tensor Index Set}
\label{ss:full_tensor}
This section focuses on the special case of a full tensor index set.
Namely, for a given vector $\vec L = (L_1, L_2, \ldots, L_d)$ we
consider the index set
$\mathcal I (\vec L) = \{\valpha \in \nset^d\::\: \alpha_i \leq L_i
\text{ for all }  i \in I\}$.
Note that in this case, $\E{\mathcal A} = S_{\vec L}$, since the sum
telescopes.  Under \textbf{Assumptions 1-3}, the following theorem
outlines the total work of the MIMC estimator when using a full tensor
index set.

\begin{theorem}[Full Tensor Work Complexity]
  \label{thm:gmlmc_ft}
  Under \textbf{Assumptions 1-3}, for $\mathcal{I}(\vec L) =
  \{\valpha\in \nset^d :\alpha_i \leq L_i\text{ for } i \in \indset
  \}$ where $L_i \in \rsetp \cup \{0\}$ for all $i \in \indset$, the
  following choice of $(L_i)_{i=1}^d$ satisfies constraint
  \eqref{eq:bias_const}:
\begin{equation}\label{eq:Liweak}
  L_i  = \frac{\log(\tol_B^{-1}) +
    \log(\mathcal{C_B})}{\log(\beta_i) w_i}\quad \text{ for all } i \in \indset,
\end{equation}
\begin{equation}
  \label{eq:fulltensor_C_Bias}
  \text{where}\qquad
\mathcal{C_B} =d \left( \prod_{j=1}^d \frac{\beta_j^{w_j} }{1-\beta_j^{-w_j}} \right).  \end{equation}
Moreover,  assuming that
\begin{equation}
  \label{eq:ft_assump_dominant}
  \sum_{i \in \indsetA \cup \indsetB} \frac{\gamma_i}{w_i} +
    \sum_{i \in \indsetC} \frac{s_i}{w_i} < 2,
\end{equation}
the optimal total work, $W(\mathcal I)$, of the MIMC estimator,
$\mathcal{A}$, subject to statistical error  constraint
\eqref{eq:stat_const} then satisfies
  \begin{equation*}
        \limsup_{\tol \downarrow 0} \frac{W(\mathcal I)}{\tol^{-2} \left( \prod_{i=1}^d \mathfrak{r}_i \right)^2} \leq
   \frac{C_\epsilon^2 Q_S \CWork}{\theta^2} \prod_{i=1}^d \mathfrak{K}_i^{-2} < \infty,
 \end{equation*}
 \begin{subequations}
   \label{eq:gmlmc_asym_cases}
\begin{equation*}\aligned
    \text{where} \qquad &\mathfrak{r}_i &&= \begin{cases}
  1 & \text{if } s_i > \gamma_i , \\
  \log(\tol^{-1}) & \text{if } s_i = \gamma_i, \\
  \tol^{\frac{ - (\gamma_i-s_i)}{2w_i} } & \text{if } s_i <\gamma_i,
\end{cases}\\
\text{and} \qquad &\mathfrak{K}_i &&= \begin{cases}
  1-\beta_i^{-\frac{s_i-\gamma_i}{2}}  & \text{if } s_i > \gamma_i , \\
  \log (\beta_i) w_i  & \text{if } s_i = \gamma_i, \\
\left( {1-\beta_i^{-\frac{\gamma_i - s_i}{2}}}\right)
\left( \frac{(1-\theta)} {\mathcal{C_B} Q_W} \right)^{\frac{\gamma_i - s_i}{2w_i}}& \text{if } s_i <\gamma_i.
\end{cases}
\endaligned\end{equation*}
\end{subequations}
\end{theorem}
\begin{proof}
  First, for convenience, we introduce the following notation for all
  $i \in \indset$:
\begin{align}
  \label{eq:rates_not}
  \overline s_i = \log(  {\beta_i}) s_i, \quad   \overline w_i = \log(
    {\beta_i}) w_i, \quad  \overline \gamma_i = \log(
    {\beta_i}) \gamma_i,
  \end{align}
  and correspondingly the following vectors:
  \begin{align}
  \label{eq:rates_not_vec}
  \overline {\vec s} = (\overline s_i)_{i \in \indset}, \quad
  \overline {\vec w} = (\overline w_i)_{i \in \indset}, \quad
  \overline {\vec \gamma} = (\overline \gamma_i)_{i \in \indset}.
  \end{align}

  Then, by \textbf{Assumption} \textbf{1},
  starting from \eqref{eq:gmlmc_biastilde_model}, we have
   $$\aligned
   \widetilde B(\mathcal I(\vec L)) &=
   \sum_{\valpha \notin \mathcal I(\vec L)} \prod_{i=1}^d
   \exp(-\overline w_i \alpha_i) \\
 &\le \sum_{i=1}^d \left\{ \sum_{\{\valpha \::\: \alpha_i >
     L_i \}} \prod_{j=1}^d
    \exp(-\overline w_j \alpha_j)  \right\} \\
   &\le
   \sum_{i=1}^d  \left\{
   {\left( \prod_{j\ne i}\frac{\exp(\overline w_j)}{\exp(\overline w_j) -1} \right)} \sum_{\alpha_i>\lfloor L_i\rfloor}\exp(-\overline w_i \alpha_i ) \right\}\\
   &\le \left( \prod_{j=1}^d\frac{\exp(\overline w_j)}{\exp(\overline w_j) -1} \right) \sum_{i=1}^d     {\exp(-\overline w_i (L_i-1))}.
   \endaligned
   $$
   Recall \eqref{eq:tolB}. Then, making each of the terms in the previous sum less than $\tol_B/d$ to satisfy
\eqref{eq:gmlmc_biastilde_model}
yields the following condition on $L_i$:
\begin{equation*}
  L_i  \geq \frac{\log(\tol_B^{-1}) +
    \log(\mathcal{C_B})}{\log(\beta_i) w_i}\quad \text{ for all } i \in \indset,
\end{equation*}
which is satisfied by \eqref{eq:Liweak}.
On the other hand, using definition \eqref{eq:gmlmc_worktilde_model},
 we have
 \begin{equation*}
 \aligned
 \workEst(\mathcal I(\vec L))
 &{=\sum_{\valpha\in \mathcal I}\prod_{i=1}^d \exp(\overline g_i\alpha_i)}\leq  \prod_{i=1}^d
   \sum_{\alpha_i=0}^{\lfloor L_i \rfloor}  \exp\left( \overline g_i
  {\alpha_i}\right) \\
 &\le
 \prod_{i\in \indsetA} \frac{1}{1-\exp(\overline g_i)} \,
 \prod_{i\in \indsetB}  {(L_i+1)}\,
 \prod_{i\in \indsetC}  \frac{\exp(\overline g_iL_i ) -
   \exp(-\overline g_i)}{1-\exp(-\overline g_i)}.
\endaligned
\end{equation*}
From here and using \eqref{eq:Liweak}, it is easy to verify that
\begin{equation}
  \label{eq:prelimwork}
  \limsup_{\tol \downarrow 0} \frac{\workEst(\mathcal I)}{\prod_{i=1}^d
    \mathfrak{r}_i} \leq \prod_{i=1}^d \mathfrak{K}_i^{-1}.
\end{equation}
Similarly, using definition \eqref{eq:gmlmc_work1_model} and \eqref{eq:Liweak}, we have
\begin{equation*} \aligned
  \workRem(\mathcal I) &= \sum_{\valpha\in \mathcal
    I}\prod_{i=1}^d \exp(\overline \gamma_i\alpha_i)\leq
  \prod_{i=1}^d   \sum_{\alpha_i=0}^{\lfloor L_i \rfloor}  \exp\left(
    \overline \gamma_i  {\alpha_i}\right) \\
 &\le \prod_{i=1}^d  \frac{\exp(\overline \gamma_i L_i ) -
   \exp(-\overline \gamma_i)}{1-\exp(-\overline \gamma_i)} \\
 &= \Order{\tol^{-\sum_{i=1}^d \frac{\gamma_i}{w_i}}}.  \,
 \endaligned
\end{equation*}
Then, due to \eqref{eq:ft_assump_dominant}, the first term in
\eqref{eq:gmlmc_work_model} dominates $\workRem(\mathcal I)$ as
$\tol \downarrow 0$.
The proof finishes by combining
\eqref{eq:prelimwork} and \eqref{eq:gmlmc_worktilde_model}.
\end{proof}

\inchangesA{ The work estimates in the previous theorem require
  the restrictive condition \eqref{eq:ft_assump_dominant} to be satisfied.
  However, we can relax this condition and obtain better work
  complexity by carefully choosing the index set, $\mathcal I$, as the
  next section shows.  }
 
\makeatletter{}\subsection{Optimal Index Sets}\label{ss:td_set}
We discuss in this section how to find optimal index sets, $\mathcal I$.
The objective is to solve the following optimization problem:
\[ \min_{ \mathcal I \subset \nset^d } W(\mathcal I)\quad \text{ such
  that }\quad \text{Bias}(\mathcal I) \leq (1-\theta)\tol. \] We choose
the number of samples according to \eqref{eq:optimal_M} and
use the upper bound of the work \eqref{eq:gmlmc_work_model} and the
upper bound of the bias \eqref{eq:bias_I}.  Moreover, we assume that
the first term in \eqref{eq:gmlmc_work_model} dominates the
second. Based on this, we instead solve the following simplified
problem:
\begin{equation}
\min_{ \mathcal I \subset \nset^d } \widetilde W(\mathcal I)\quad
\text{ such that }\quad \widetilde B(\mathcal I) \leq \tol_B
\label{eq:optI_problem}
\end{equation}
\inchangesA{ to get a quasi-optimal index set, $\mathcal I$.  Here,
  $\widetilde W$ is defined in \eqref{eq:gmlmc_worktilde_model} and
  $\widetilde B$ is defined in \eqref{eq:gmlmc_biastilde_model}.  In
  what follows, we discuss how to solve the optimization problem
  \eqref{eq:optI_problem}.  In the rest of this section, with a slight
  abuse of terminology, we refer to the objective $\widetilde W$ as
  the ``work'' and the constraint function $\widetilde B$ as the
  ``error''.}

Similar to \cite{nobile2014convergence}, the optimization problem
\eqref{eq:optI_problem} can be recast into a knapsack
problem where a ``profit'' indicator is assigned to each index and
only the most profitable indices are added to $\mathcal I$.  Let us
define the profit, $\mathcal{P}_\valpha =
\frac{\varepsilon_\valpha}{\varpi_\valpha}$, of a multi-index,
$\valpha$, in terms of its error contribution, denoted here by
$\varepsilon_\valpha$, and its work contribution, denoted here by
$\varpi_\valpha$.  Moreover, define the total error associated with an
index set, $\mathcal{I}$, as
\[  \mathfrak E(\mathcal{I}) = \sum_{\valpha\notin\mathcal{I}} \varepsilon_\valpha \]
and the corresponding total work as
\[  \mathfrak W(\mathcal{I}) = \sum_{\valpha\in\mathcal{I}} \varpi_\valpha.  \]
Intuitively, we may think of $\mathfrak E(\mathcal{I})$ as a sharp upper bound for
$\widetilde B(\mathcal I)$ and $\mathfrak W(\mathcal I)$ as
a correspondingly sharp lower bound for $\widetilde W(\mathcal I)$.
  Then, we can show the following optimality result with respect to
  $\mathfrak E(\mathcal{I})$  and  $\mathfrak W(\mathcal I)$, namely:
\begin{lemma}[Optimal profit sets]\label{lem:optimality}
The set $\mathcal{I}(\nu) = \{\valpha \in \nset^d:  \mathcal{P}_\valpha \ge \nu\}$ is optimal in the
sense that any other set, $\tilde{\mathcal{I}}$,
with  smaller work, $\mathfrak W(\tilde{\mathcal{I}})< \mathfrak W(\mathcal{I}(\nu))$, leads to a larger error, $\mathfrak E(\tilde{\mathcal{I}})> \mathfrak E(\mathcal{I}(\nu))$.
\end{lemma}

\begin{proof}
We have that for any $\valpha \in \mathcal{I}(\nu)$ and $\hat \valpha \notin {\mathcal{I}(\nu)}$
\[
\mathcal{P}_{\valpha} \geq \nu \qquad \text{and}\qquad
\mathcal{P}_{\hat \valpha} < \nu.
\]
  Now, take  an arbitrary index set, $\tilde{\mathcal{I}}$, such that
  $\mathfrak W(\tilde{\mathcal{I}}) < \mathfrak W(\mathcal{I}(\nu))$
  and divide $\nset^d$ into the following disjoint sets:
\begin{align*}
  \mathcal J_1 &= \mathcal I(\nu) \cap \tilde{\mathcal I}^c,\qquad
  &\mathcal J_2 &= \mathcal I(\nu) \cap \tilde{ \mathcal I},\\
  \mathcal J_3 &=  \mathcal I(\nu)^c \cap \tilde{ \mathcal I},\qquad
  &\mathcal J_4 &= \mathcal I(\nu)^c \cap \tilde{ \mathcal I}^c,
\end{align*}
where $\mathcal I(\nu)^c$ is the complement of the set $\mathcal I(\nu)$. Then,
\[
\mathfrak W(\mathcal I(\nu)) - \mathfrak W(\tilde{ \mathcal I}) =
\sum_{\valpha \in \mathcal J_1 \cup \mathcal J_2} \varpi_\valpha -
\sum_{\valpha \in J_2 \cup J_3} \varpi_\valpha = \sum_{\valpha \in
  \mathcal J_1} \varpi_\valpha - \sum_{\valpha \in \mathcal J_3}
\varpi_\valpha > 0,
\]
and
\[
\mathfrak E(\mathcal I(\nu)) - \mathfrak E(\tilde{ \mathcal I}) =
\sum_{\valpha \in \mathcal J_3 \cup \mathcal J_4} \varepsilon_\valpha -
\sum_{\valpha \in \mathcal J_1 \cup \mathcal J_4} \varepsilon_\valpha =
\sum_{\valpha \in \mathcal J_3} \mathcal{P}_\valpha \varpi_\valpha - \sum_{\valpha \in
  \mathcal J_1} \mathcal P_\valpha \varpi_\valpha.
\]
Then,
\[\mathfrak E(\mathcal I(\nu)) - \mathfrak E(\tilde{ \mathcal I}) \leq \nu \left( \sum_{\valpha \in \mathcal J_3} \varpi_\valpha - \sum_{\valpha \in
  \mathcal J_1} \mathcal \varpi_\valpha \right) < 0.\]
\end{proof}

For MIMC, under \textbf{Assumptions 1-3}, $\varepsilon_\valpha$ can be
taken to be the bias contribution of the term $\estS_\valpha$, i.e.,
$\varepsilon_\valpha = E_\valpha$. Additionally, the work contribution
can also be taken as $\varpi_\valpha = \sqrt{V_\valpha
  W_\valpha}$.     Using the estimates in \textbf{Assumptions 1-3} as
sharp approximations to their counterparts, the profits in our problem
are approximated correspondingly by
$$
\mathcal{P}_\valpha \approx C_P \prod_{i=1}^d  e^{-\alpha_i\log(\beta_i) (w_i+\frac{\gamma_i-s_i}{2})},
$$
for some constant $C_P > 0$.
Therefore, ordering the profits according to level sets as in Lemma
\ref{lem:optimality}, yields optimal index sets of multi indices that
are of anisotropic total degree (TD) type.
 Let us introduce strictly positive normalized weights defined by
\begin{equation}
  \label{eq:optimal_weights}
  \aligned
\delta_i &= \frac{\log(\beta_i)
  (w_i+\frac{\gamma_i-s_i}{2})}{C_\vdelta},\quad\text{for all } \,i \in \indset, \\
\text{where}\qquad C_\vdelta &= \sum_{j=1}^d \log(\beta_j)
(w_j+\frac{\gamma_j-s_j}{2}).
\endaligned
\end{equation}
Observe that
\begin{equation}
  \label{eq:delta_cond}
  \sum_{i \in \indset} \delta_i = 1 \quad \text{and} \quad  0<\delta_i\le 1,
\end{equation}
since $s_i\le 2w_i$ and $\gamma_i>0$ by assumption.
Then, for $L=0,1,\ldots$, we introduce a family of TD index sets:
\begin{equation}\label{eq:idxset_anisoTD}
\mathcal{I}_\vdelta(L) = \{\valpha\in \nset^d: \valpha \cdot \vdelta = \sum_{i=1}^d \delta_i\alpha_i \le L\}.
\end{equation}
In our numerical example, presented in Section~\ref{s:res},
Figure~\ref{fig:mimc_contours} suggests that the TD index set is indeed the
optimal index set in this case.

The current section continues by first considering a general vector of
weights, $\vdelta$, that satisfies only \eqref{eq:delta_cond}. We find
a value of $L$ that satisfies the bias constraint in
Lemma~\ref{lem:gmlmc_td_L_gen} then derive the resulting
computational complexity in Lemma~\ref{lem:gmlmc_td_work_gen}. Next,
we present our main result in Theorem~\ref{thm:gmlmc_td_opt} when
using the optimal weights of \eqref{eq:optimal_weights}. Finally, we
conclude this section with a few remarks about special cases. In the
following theorems, given a general vector of weights, $\vdelta$, we
introduce the following notation:
\begin{subequations}
  \label{eq:max_rates}
  \begin{align}
    \eta &= \min_{i \in \indset} \frac{\log(\beta_i)w_i}{\delta_i}, &
    \mathfrak{e} &= \#\{ i \in \indset : \frac{\log(\beta_i)w_i}{\delta_i} = \eta\},\\
    \Gamma &= \max_{i \in \indset}  \frac{\log(\beta_i)\gamma_i}{\delta_i}, &
    \mathfrak{g} &= \#\{ i \in \indset : \frac{\log(\beta_i)\gamma_i}{\delta_i} = \Gamma\}, \\
    \chi &= \max_{i \in \indset} \frac{\log(\beta_i)(\gamma_i - s_i)}{2\delta_i}, &
  \mathfrak{x} &= \#\{ i \in \indset : \frac{\log(\beta_i)(\gamma_i - s_i)}{2\delta_i} = \chi\},\\
    \zeta &= \max_{i \in \indset} \frac{\gamma_i - s_i}{2w_i}, &
    \mathfrak{z} &= \#\{ i \in \indset : \frac{\gamma_i - s_i}{2w_i} =
    \zeta\},\\
    \xi &= \min_{i \in I} \frac{2 w_i - s_i}{\gamma_i}.
  \end{align}
\end{subequations}
\begin{lemma}[$L$ of MIMC with general $\vdelta$]
  \label{lem:gmlmc_td_L_gen}
  Consider the multi-index sets $\mathcal{I}_\vdelta(L) = \{ \valpha \in \nset^d :
  \vdelta \cdot \valpha \leq L\}$ with given weights $\vdelta \in
  \rsetp^d$ satisfying \eqref{eq:delta_cond} and $L \in \rset_+ \cup \{0\}$.
If \textbf{Assumption 1} holds, then, to satisfy the following bias inequality,
\begin{equation}
  \label{eq:bias_asymb}
\lim_{\tol \downarrow 0} \frac{\widetilde B(\mathcal I_\vdelta(L))}{\tol_B} \leq 1,
\end{equation}
with $\widetilde B(\mathcal I_\vdelta(L))$ as defined in
\eqref{eq:gmlmc_biastilde_model}, we can take $L$ as follows:
\begin{equation}
  \label{eq:L_td_set_general}
L = \frac{1}{\eta }\left(  \log(\tol_B^{-1}) + \left( \mathfrak{e}-1 \right)
  \log\left( \frac{1}{\eta }
    \log(\tol_B^{-1}) \right) + \log({C_{\textnormal{Bias}}}) \right).
\end{equation}
Here, ${C_{\textnormal{Bias}}}$ is given by
\begin{equation}\label{eq:CBias_td}
  \aligned
  {C_{\textnormal{Bias}}} = \exp(|\overline{\vec w}|)
  \left(\prod_{i=1}^d \delta_i^{-1}\right) \mathfrak{C_B}\left( \left(
    \delta_i^{-1} \log(\beta_i) w_i \right)_{i=1}^d \right),
\endaligned
\end{equation}
and  $\mathfrak{C_B}$ is defined in \eqref{eq:exp_int_bound_CB}.
\end{lemma}
\begin{proof}
  For small enough $\tol$, such that $L\geq 1$ in
  \eqref{eq:L_td_set_general}, we have, using Lemma \ref{lem:bias_bound}
\begin{align*}
  \widetilde B(\mathcal I_\vdelta(L)) &= \sumlim{\valpha \in
    \nset^d}{\valpha \cdot \vdelta > L}
  \exp(-\overline{\vec w} \cdot \valpha)\\
  &\leq \intlim{\vec x \in \rsetp^d}{\vec x \cdot \vdelta \geq L}
  \exp(- \overline{\vec w} \cdot (\vec x - \vec 1) )\: \textnormal{d}\vec{x} \\
  &= \exp(|\overline{\vec{w}}|) \left(\prod_{i=1}^d \delta_i^{-1} \right)
  \intlim{\vec x \in \rsetp^d}{|\vec x| \geq L} \exp\left(-\sum_{i=1}^d
  \delta_{i}^{-1} x_i\overline w_i \right)\:
  \textnormal{d}\vec{x} \\
  &\leq \exp(|\overline{\vec{w}}|) \left(\prod_{i=1}^d \delta_i^{-1}
  \right)
  \mathfrak{C_{B}}\left(\left(
    \delta_i^{-1} \log(\beta_i) w_i \right)_{i=1}^d \right) \exp(-L \eta)L^{\mathfrak{e}-1}.
\end{align*}
Substituting $L$ from  $\eqref{eq:L_td_set_general}$ and taking the
limit $\tol\downarrow 0$ yields
\[
\aligned
&\lim_{\tol \downarrow 0} \frac{\widetilde B(\mathcal
  I_\vdelta)}{\tol_B} \\
&\leq
\lim_{\tol \downarrow 0} \left( 1 +  \frac{ \left( \mathfrak{e}-1 \right) \log\left(\frac{1}{\eta}
  \log(\tol_B^{-1}) \right) + \log({C_{\textnormal{Bias}} })}{\log\left( \tol_B^{-1}
\right)} \right)^{\mathfrak{e}-1} = 1,
\endaligned
\]
which finishes the proof.
\end{proof}

\begin{lemma}[Work estimate of MIMC with general $\vdelta$]
  \label{lem:gmlmc_td_work_gen}
   Consider the multi-index sets $\mathcal{I}_\vdelta(L) = \{ \valpha \in \nset^d :
  \vdelta \cdot \valpha \leq L\}$ with given weights $\vdelta \in
  \rsetp^d$ satisfying \eqref{eq:delta_cond} and
  take  $L$ as \eqref{eq:L_td_set_general}.
  Under \textbf{Assumptions 1-3}, the bias inequality
  \eqref{eq:bias_asymb} is satisfied and the total work, $W(\mathcal
  I_\vdelta)$, of the MIMC estimator, $\mathcal{A}$, subject to
  constraint \eqref{eq:stat_const} satisfies
  \begin{equation}
    \label{eq:work_td_limit_caseA}
    \limsup_{\tol \downarrow 0}
    \frac{W(\mathcal I_\vdelta)}{\tol^{-2 \left( 1 + \max\left\{0,
          \frac{\chi}{\eta},  \frac{\Gamma - 2\eta}{2\eta}
          \right\}  \right)} \left(
        \log\left(\tol^{-1} \right) \right)^{\mathfrak{p}}} \leq
    \CWork \mathcal{C} < \infty,
  \end{equation}
  where
\begin{itemize}
\item[\textbf{Case A)}]
  $ \displaystyle \aligned[t]
    \text{if } \chi \leq 0 \text{ and }&\text{either} \qquad & \Gamma< 2 \eta&, \\
    &\text{ or} & \Gamma = 2 \eta
  &\quad \text{ and } \quad 2\mathfrak{e} +\mathfrak{g} < 2d_2 +
  3,
  \endaligned$ \\ then $\mathfrak{p} = 2d_2$ and $ \mathcal{C} = {Q_S
    C_\epsilon^2 }{\theta^{-2} C_A^{-2}}$,
    \begin{equation}
      \label{eq:CW_caseA}
      \aligned
      \text{where} \quad
      {C_A} &= \left( \prod_{i \in \indsetA} \left(
          1-\beta_i^{-\frac{s_i-\gamma_i}{2}} \right) \right) \left(
        {\prod_{j\in \indsetB} \delta_{j}} \right) \eta^{d_2} d_2! .
      \endaligned
    \end{equation}
  \item[\textbf{Case B)}]
    $ \displaystyle \aligned[t]
        \text{If } \chi > 0 \text{ and }&\text{either} \qquad & \Gamma < 2 \eta + 2\chi&, \\
        &\text{or} & \Gamma = 2 \eta + 2\chi &\quad \text{ and
    } \quad 2 \mathfrak{e} + \mathfrak{g} < 2\mathfrak{x} +
    1,
    \endaligned$
    \\ then, $\mathfrak{p} = 2 \left(
      \mathfrak{x}-1+{\frac{(\mathfrak{e}-1) \chi}{\eta}}
    \right)$ and $ \mathcal{C} = {Q_S C_\epsilon^2
    }{\theta^{-2} C_B^{-2}}$, where
    \begin{equation}
      \label{eq:CW_caseB}
      \aligned      {C_B} &=
      \frac{      \prod_{i \in \indsetA} \left(1-\beta_i^{-\frac{s_i-\gamma_i}{2}} \right)}
      {      \prod_{i \in \indsetBC} \delta_i^{-1}} \cdot
      \frac{\exp\left(-\chi \right)  \eta^{\frac{\mathfrak{p}}{2}}}
      {\mathfrak{C_W} \left( \left( \delta_i^{-1} \overline g_i \right)_{i \in
            \indsetBC} \right)} \cdot
      \left( \frac{1-\theta}{{C_{\textnormal{Bias}}} Q_W}\right)^{\frac{\chi}{\eta}} .
      \endaligned
    \end{equation}
\item[\textbf{Case C)}]
\inchangesA{  $ \displaystyle \aligned[t]
  \text{if } \chi \leq 0 \text{ and }&\text{either} \qquad & \Gamma > 2 \eta&, \\
  &\text{ or} & \Gamma = 2 \eta
  &\quad \text{ and } \quad 2\mathfrak{e} +\mathfrak{g} \geq 2d_2 +
  3,
  \endaligned$ \\
then $\mathfrak{p} = \mathfrak{g}-1 + (\mathfrak{e}-1)
   \frac{\Gamma}{\eta}$ and $\mathcal{C} = \mathscr{I}_C
   {Q_S C_\epsilon^2 }{\theta^{-2} {C_A}^{-2}} + C_R^{-1}$, where
       \begin{align}
         \label{eq:work_rem_constant}
         C_R &= \left(\prod_{i \in \indset} \delta_i \right)
         \frac{\exp\left(-\Gamma \right) \eta^{\mathfrak{p}} }
         {\mathfrak{C_W}( \left( \delta_i^{-1} \overline \gamma_i
           \right)_{i \in \indset})}
         \left( \frac{1-\theta}{C_{\textnormal{Bias}} Q_W}\right)^{\frac{\Gamma}{\eta}} ,\\
         \label{eq:ind_caseC}
         \text{and} \qquad \mathscr{I}_C &= \begin{cases}
           1 & \text{if } {\Gamma = 2\eta \text{ and } 2\mathfrak{e} +\mathfrak{g} = 2d_2 + 3}, \\
           0 & \text{if } {\Gamma > 2\eta \text{ or } 2\mathfrak{e} +\mathfrak{g} > 2d_2 + 3}.
         \end{cases}
       \end{align}}
\item[\textbf{Case D)}]
\inchangesA{
    $ \displaystyle \aligned[t]
        \text{If } \chi > 0 \text{ and }&\text{either} \qquad & \Gamma > 2 \eta + 2\chi&, \\
        &\text{or} & \Gamma = 2 \eta + 2\chi &\quad \text{ and
    } \quad 2 \mathfrak{e} + \mathfrak{g} \geq 2\mathfrak{x} +
    1,
    \endaligned$
    \\          then $\mathfrak{p} = \mathfrak{g}-1 + (\mathfrak{e}-1)
         \frac{\Gamma}{\eta}$ and $\mathcal{C} = \mathscr{I}_D
   {Q_S C_\epsilon^2 }{\theta^{-2} {C_B}^{-2}} + C_R^{-1}$, where
\begin{equation}
  \label{eq:ind_caseD}
  \mathscr{I}_D = \begin{cases} 1 & \text{if } \Gamma = 2\eta + 2\chi
    \text{ and }  2 \mathfrak{e} + \mathfrak{g} = 2\mathfrak{x}+1, \\
    0 & \text{if }\Gamma > 2\eta + 2\chi
    \text{ or } 2 \mathfrak{e} + \mathfrak{g} >
    2\mathfrak{x}+1.
         \end{cases}
       \end{equation}}
\end{itemize}
\end{lemma}
\begin{proof}
  First note that \eqref{eq:bias_asymb} is satisfied due to
  Lemma \ref{lem:gmlmc_td_L_gen}. Now, we need to bound the work in
  \eqref{eq:gmlmc_work_model}. We start with the term
  $\workRem(\mathcal I)$. Using Lemma \ref{lem:work_bound}, we have:
  \[
  \aligned \workRem(\mathcal I_{\vdelta}(L)) &= \sum_{\alpha \in
    \mathcal
    I_{\vdelta}(L)} \exp(\overline{\vec{\gamma}} \cdot \valpha) \\
  &\leq \left( \prod_{i \in I} \delta_i^{-1} \right) \intlim{\vec x \in
    \rsetp^{d}}{|\vec x| \le
    L+1} \exp\left(\sum_{i=1}^d \delta_i^{-1}\overline{\gamma}_i x_i \right)\, \textnormal{d}\vec x  \\
  &\le  \left( \prod_{i \in I} \delta_i^{-1} \right)
  \mathfrak{C_W} \left(\left( \delta_i^{-1} \overline \gamma_i\right)_{i
      \in \indset} \right)  \exp(\Gamma (L+1)) \left( L+1
  \right)^{\mathfrak{g}-1}.
  \endaligned
  \]
  Then, substituting $L$ from \eqref{eq:L_td_set_general} and taking the
limit $\tol\downarrow 0$ yield
  \begin{equation}
    \label{eq:work_rem_rate}
    \lim_{\tol \downarrow 0} \frac{\workRem(\mathcal I_{\vdelta}(L))}{\tol^{\frac{-\Gamma}{\eta}} \left( \log(\tol^{-1})
      \right)^{\mathfrak{m}}} = {C_R^{-1}} ,
  \end{equation}
  where $\mathfrak{m} = \mathfrak{g} - 1
  +(\mathfrak{e}-1)\frac{\Gamma}{\eta}$.

  Next, we focus on the term $\widetilde W(\mathcal I_\vdelta(L))$.
  Define $\tilde \vdelta_1 = \left( \delta_i \right)_{i \in \indsetA}$
 to be the entries of $\vdelta$
  corresponding to the index set, $\indsetA$, introduced in \eqref{eq:ind_sets}. Similarly define $\hat
  \vdelta$ corresponding to $\indsetBC$.
  Then, starting from \eqref{eq:gmlmc_worktilde_model}, we have
  \begin{equation}
    \label{eq:td_w_est_gen}
    \aligned
    \widetilde W(\mathcal I_\vdelta(L)) &= \sum_{\valpha \in \mathcal
      I_\vdelta(L)} \exp(\overline{\vec g} \cdot \valpha) \\
    &\leq  \underbrace{\left(\sum_{\valpha \in \nset^{d_1}, \valpha \cdot \tilde \vdelta_1 \leq
          L} \exp\left( \sum_{i \in \indsetA} \overline g_i \alpha_i \right) \right)}_{\eqdef P_1} \underbrace{\left(
        \sum_{\valpha \in
          \nset^{\hat d}, \valpha \cdot \hat{\vdelta} \leq L}
        \exp\left( \sum_{i \in \indsetBC}\overline g_i \alpha_i
        \right) \right)}_{\eqdef \hat P}.
    \endaligned
  \end{equation}
  Now, observe that for the term $P_1$, since $\overline g_j<0$ for all $j\in \indsetA$,  we have
    \begin{equation}
    \label{eq:p1_est_gen}
    P_1  \le \frac{1}{\prod_{j \in \indsetA} (1-\exp({\overline g_j })) }.
  \end{equation}

  For the term $\hat P$ in \eqref{eq:td_w_est_gen}, we distinguish between two cases:
  \begin{itemize}
  \item
    If $\chi$ from \eqref{eq:max_rates} satisfies $\chi \leq 0$, then
    $\max_i \overline g_i \leq 0$, $\indsetC = \emptyset$ and $\indsetBC = \indsetB$. Thus,
    since $\overline g_j=0$ for all $j\in \indsetB$, we have
    \[
    \aligned
    \hat P = &
    \sumlim{\valpha \in \nset^{d_2}}{\tilde\vdelta_2 \cdot\valpha \le L}   1 \\
    \le &
    \intlim{\vec x \in \rsetp^{d_2}}{\vec x\cdot\tilde\vdelta_2 \le
      L+|\tilde\vdelta_2|} 1\, \textnormal{d}\vec x  \\
    \le &
    \frac{1}{\prod_{j\in \indsetB} \delta_{j}}  \intlim{\vec y \in \rsetp^{d_2}}{|\vec y| \le L+|\tilde\vdelta_2|} 1\, \textnormal{d}\vec y \\
    \leq & \frac{1}{\prod_{j\in \indsetB} \delta_{j}}  \frac{(L+1)^{d_2}}{d_2!}.
    \endaligned
    \]
    Combining the previous inequality with
    \eqref{eq:td_w_est_gen}, \eqref{eq:p1_est_gen},
    and \eqref{eq:L_td_set_general} and
    taking the limit of the resulting expression as $\tol \downarrow 0$ yield
    \begin{equation}
      \label{eq:caseI_worktilde_limit}
      \limsup_{\tol \downarrow 0}
      \frac{\left( \workEst(\mathcal I_\vdelta) \right)^2}{\left(
          \log\left(\tol^{-1} \right) \right)^{2d_2}} \leq
      { C_A^{-1}}.
    \end{equation}
              \item
    If $\chi > 0$, then using the identity of Lemma
    \ref{lem:work_bound} yields
    \begin{align*}
      \hat P &\leq \intlim{\vec x \in \rsetp^{\hat d} }{\vec x \cdot
        \hat \vdelta \leq L + |\hat \vdelta|} \exp\left(
      \sum_{i \in \indsetBC} \overline g_i x_i \right) \textnormal{d}\vec x \\
      &\leq \left(\prod_{i \in \indsetBC} \delta_i^{-1} \right)
      {\mathfrak{C_W} \left( \left( \delta_i^{-1} \overline g_i \right)_{i \in
            \indsetBC} \right)} \exp\left(\chi L\right) \left( L+1
      \right)^{\mathfrak{x}-1}.
    \end{align*}
    Combining the previous inequality with
    \eqref{eq:td_w_est_gen}, \eqref{eq:p1_est_gen},
    \eqref{eq:gmlmc_worktilde_model} and \eqref{eq:L_td_set_general} and
    taking the limit of the resulting expression as $\tol \downarrow 0$ yield
    \begin{equation}
      \label{eq:caseII_worktilde_limit}
      \limsup_{\tol \downarrow 0}
      \frac{\left( \workEst(\mathcal I_\vdelta) \right)^2}{\tol^{-\frac{2\chi}{\eta}} \left( \log\left(
            \tol^{-1} \right)\right)^\mathfrak{j}
      } \leq
      {C_B^{-1}}.
    \end{equation}
    where $\mathfrak j = 2 \left(
      \mathfrak{x}-1+{\frac{(\mathfrak{e}-1) \chi}{\eta}}
    \right)$.
              \end{itemize}

  Now we are ready to prove the different cases. In this lemma,
  \textbf{Cases A} and \textbf{B} are the cases when the first term in
  \eqref{eq:gmlmc_work_model} dominates the second and the proof
  follows by substituting \eqref{eq:caseI_worktilde_limit} or
  \eqref{eq:caseII_worktilde_limit} in the right-hand side
  \eqref{eq:work_td_limit_caseA}.  On the
  other hand, in \textbf{Cases C} and \textbf{D}, the second term in
  \eqref{eq:gmlmc_work_model} either dominates the first or has
  the same order. In these cases, the proof is done by substituting
  \eqref{eq:work_rem_rate} in the right-hand side
  \eqref{eq:work_td_limit_caseA}.
                              \end{proof}

We are now ready to state and prove our main result, which is a special
case of the previous lemma when we make the specific choice of
$\vdelta$ as in \eqref{eq:optimal_weights}.

\begin{theorem}[Work estimate with optimal weights]
  \label{thm:gmlmc_td_opt}
  Let the approximation index set be $\mathcal{I_\vdelta}(L) = \{
  \valpha \in \nset^d : \vdelta \cdot \valpha \leq L\}$ for $\vdelta
  \in (0,1]^d$ given by~\eqref{eq:optimal_weights} and take $L$ as
  \eqref{eq:L_td_set_general}.
     Under \textbf{Assumptions 1-3}, the bias inequality
  \eqref{eq:bias_asymb} is satisfied and the total work, $W(\mathcal
  I_\vdelta)$, of the MIMC estimator, $\mathcal{A}$, subject to
  constraint \eqref{eq:stat_const} satisfies the following
  \begin{equation*}
    \limsup_{\tol \downarrow 0}
    \frac{W(\mathcal I_\vdelta)}{\tol^{-2\left(1 + \max\left(0, \zeta
                      \right)\right)} \left(
        \log\left(\tol^{-1} \right) \right)^{\mathfrak{p}}} \leq
    \CWork \mathcal{C} < \infty,
  \end{equation*}
  where
\begin{itemize}
\item[\textbf{Case A)}]
$\displaystyle\aligned[t]
                \text{If }&\text{either} \qquad &
        \zeta \leq 0& \quad \text{ and } \quad \zeta < \xi, \\
                &\text{or} & \zeta = \xi = 0 &\quad \text{ and } \quad d \leq
        2,
        \endaligned$
  \\
  then $\mathfrak{p} = 2 d_2$ and
$ \mathcal C = {Q_S C_\epsilon^2}{\theta^{-2} C_A^{-2}}$
        where ${C_A}$ is defined in \eqref{eq:CW_caseA}.
      \item[\textbf{Case B)}] if $\zeta > 0$ and $\xi > 0$, then $\mathfrak{p}
        = 2(\mathfrak{z} - 1) (\zeta + 1)$ and $ \mathcal C =
        {Q_S C_\epsilon^2 }{\theta^{-2} C_B^{-2}}, $ where ${C_B}$ is
        defined in \eqref{eq:CW_caseB}.
      \item[\textbf{Case C)}] If $\zeta = \xi = 0$ and
        $d > 2$ then
 $\mathfrak{p} = 2d_2 + d - 3$ and we have
$
{ \mathcal C =  \mathscr{I}_C {Q_S C_\epsilon^2}{\theta^{-2}
   C_A^{-2}} + {C_R^{-1}},}
 $
       where $C_R$ is defined in \eqref{eq:work_rem_constant} and
       $\mathscr{I}_C$ is defined in \eqref{eq:ind_caseC} and
       simplifies to:
       \[ \mathscr{I}_C = \begin{cases}
         1 & \text{if } {d = 3}, \\
         0 & \text{if } {d > 3}.
\end{cases}
\]
\item[\textbf{Case D)}] if $\zeta > 0$ and $\xi = 0$, then $\mathfrak{p} = d-1
  + 2(\mathfrak{z}-1)(1+\zeta)$ and $\mathcal C = \mathscr{I}_D
    {Q_S C_\epsilon^2}{\theta^{-2} C_B^{-2}} + C_R^{-1}$ where
  $\mathscr{I}_D$ is defined in \eqref{eq:ind_caseD} and simplifies
  to:
\[ \mathscr{I}_D = \begin{cases}
  1 & \text{if } {d = 1}, \\
  0 & \text{if } {d > 1}.
\end{cases}
\]
\end{itemize}
\end{theorem}
\begin{proof}
                First, recall that, due to \eqref{eq:optimal_weights},
  $ \delta_j = \frac{\overline w_j+\overline g_j}{C_\vdelta} $. Then,
  using \eqref{eq:max_rates}, we have
  \[
  \aligned
  \mathfrak{e} &= \#\{i \in \indset: \delta_i^{-1} \overline w_i =
  \min_{j \in \indset} \delta_j^{-1} \overline w_j
  \} \\
   &= \#\{i \in \indset: \frac{\overline w_i}{\overline w_i +
     \overline g_i} =
 \min_{j \in \indset}  \frac{\overline w_j}{\overline w_j + \overline g_j}
 \} \\
    &= \#\{i \in \indset: 1 + \frac{\overline g_i}{\overline w_i} =
    1 + \max_{j \in \indset} \frac{\overline g_j}{\overline w_j}
  \} = \mathfrak{z}
  \endaligned
  \]
  \[
  \aligned
  \mathfrak{x} &= \#\{i \in \indset: \delta_i^{-1} \overline g_i =
  \max_{j \in \indset} \delta_j^{-1} \overline g_j
  \} \\
   &= \#\{i \in \indset: \frac{\overline g_i}{\overline w_i +
     \overline g_i} =
 \max_{j \in \indset}  \frac{\overline g_j}{\overline w_j + \overline g_j}
 \} \\
    &= \#\{i \in \indset: 1 + \frac{\overline w_i}{\overline g_i} =
    1 + \min_{j \in \indset} \frac{\overline w_j}{\overline g_j}
  \} = \mathfrak{z}.
  \endaligned
  \]
  Similarly, we can show that $\mathfrak{g}=d$ when $\xi=0$.
Next, observe that, on one hand, by setting $\sigma_j=\overline g_j / \overline w_j$, we have
\begin{equation}
\aligned
\frac{1}{\eta} =
\max_{j \in \indset} \frac{\delta_j}{\overline w_j} =&
\frac{1}{C_\vdelta} \max_{ j \in \indset}\left(1+\frac{\overline g_j}{\overline w_j}\right)  =&
\frac{1}{C_\vdelta}\left(1+\max_{ j \in \indset} \sigma_j\right)
\endaligned\label{eq:optimal_1_eta}
\end{equation}
and, on the other hand, we have
\begin{equation}
\chi = \max_{j \in \indset} \frac{\overline g_j}{\delta_j} =
C_\vdelta \max_{j \in \indset} \frac{\sigma_j}{1+\sigma_j} =
C_\vdelta \frac{\max_{j \in \indset} \sigma_j}{1+\max_{j \in \indset}
  \sigma_j},
\label{eq:optimal_chi}
\end{equation}
since $f(x)=x/(1+x)$ is a monotone increasing function.
Thus, from \eqref{eq:optimal_1_eta} and \eqref{eq:optimal_chi}, we conclude that
\[ \frac{\chi}{\eta} = \max_{i \in \indset} \sigma_i = \zeta.\]
Hence, $\chi \geq 0$ if an only if $ \zeta \geq 0$.
Moreover, using
a similar calculation, we easily see that
\[\Gamma = \frac{2 C_\vdelta}{1+\xi} \quad\text{and}\quad
\eta = \frac{C_\vdelta}{1+\zeta}
\quad\text{so that}\quad
\frac{\Gamma}{2\eta} =
\frac{1+\zeta}{1+\xi}.\]
 Also, if $\zeta \leq 0$, then $ \zeta \leq 0 \leq \xi$
 since, for all $i \in \indset$, we have $s_i \leq 2 w_i$ by
\textbf{Assumptions 1-2}.
 On the other hand, if $\zeta > 0$, then
 ${\zeta-\xi} \leq \zeta({1 + \xi})$.
 In any case, for all $\zeta$, we have
 \begin{equation}
   \frac{\Gamma - 2\eta}{2\eta} = \frac{\zeta-\xi}{1 + \xi} \leq
 \max(0,\zeta).\label{eq:4}
\end{equation}
Substituting \eqref{eq:4}, $\mathfrak{z} =
\mathfrak{e} = \mathfrak{x}$ and $\zeta = {\chi}/{\eta}$ in Lemma
\ref{lem:gmlmc_td_work_gen} and
noting that $d_2 = \mathfrak{z}$ if $\zeta = 0$ and  $\mathfrak{g} = d$ if $\xi = 0$ yield the stated results in
this theorem. In particular:
\begin{itemize}
\item[\textbf{Case A)}] $\chi \leq 0$ and $\Gamma < 2\eta \Rightarrow
  \zeta \leq 0 $ and $\zeta < \xi$. On the other hand, $\chi \leq 0$ and
  $\Gamma = 2\eta \Rightarrow \zeta \leq 0$ and $\zeta = \xi \Rightarrow
  \zeta = \xi = 0$ since $\xi \geq 0$. Moreover, in the latter case, $\mathfrak{e} =
  \mathfrak{z} = d_2$ and $\mathfrak{g} = d$. Therefore, $2 \mathfrak{e} +
  \mathfrak{g} < 2 d_2 + 3 \Rightarrow d \leq 2$.
 \item[\textbf{Case B)}] $\chi > 0$ and $\Gamma < 2 \eta + 2 \chi \Rightarrow
   \zeta > 0$ and $\xi > 0$. On the other hand, $\chi > 0$ and $\Gamma
   = 2\eta + 2\chi \Rightarrow \zeta > 0$  and $\xi = 0$. Moreover, in
   the latter case, since $\mathfrak{e} = \mathfrak{x} = \mathfrak{z}$
   and $\mathfrak{g} = d$, then $2 \mathfrak{e} + \mathfrak{g} < 2
   \mathfrak{x} + 1 \Rightarrow d < 1$, which is always false since $d
   \geq 1$.
\end{itemize}
Other cases can be proved similarly.
\end{proof}

\begin{remark}[On isotropic directions]\label{rem:iso_work}
  Of particular interest is the case when ${\gamma_i = \gamma}$, ${s_i=s}$,
  ${w_i = w}$ and ${\beta_i = \beta}$ for all $i \in \indset$ and for
  positive constants $\gamma, s, w$ and $\beta$.  In this case, we
  have asymptotically as $\tol \to 0$,
  \begin{subequations}
    \begin{align} \label{eq:gmlmc_iso_ft_work}
      \begin{array}{c}
        \text{Work of MIMC with}\\
        \text{a full tensor}\\
        \text{index set}
      \end{array} = \begin{cases}
        \Order{\tol^{-2}}, & s > \gamma, \\
        \Order{\tol^{-2} \left(\log(\tol^{-1})\right)^{2d}}, & s = \gamma, \\
        \Order{\tol^{-\left(2 + \frac{d(\gamma-s)}{w}\right) }},  & s < \gamma,
      \end{cases}\\
\inchangesA{
      \begin{array}{c}
        \text{Work of MIMC}\\
        \text{with an optimal} \\
        \text{TD index set}
      \end{array} = \begin{cases}
        \Order{\tol^{-2}}, & s > \gamma, \\
        \Order{\tol^{-2} \left(\log(\tol^{-1})\right)^{2d}}, & s = \gamma, \\
        \Order{\tol^{-\left(2 + \frac{\gamma-s}{w}\right) } \left(\log(\tol^{-1})\right)^{\mathfrak{p}}},  & s < \gamma,
      \end{cases}
}    \end{align}
    where $\mathfrak{p}$ can be found in Theorem
    \ref{thm:gmlmc_td_opt}.
On the other hand, we have \cite{cst13}
    \begin{align} \label{eq:mlmc_iso_work} \text{Work of MLMC}
      = \begin{cases}
        \Order{\tol^{-2}}, & s > d\gamma, \\
        \Order{\tol^{-2} \left(\log(\tol^{-1})\right)^2}, & s = d\gamma, \\
        \Order{\tol^{-\left(2 +\frac{d\gamma-s}{w}\right)}},  & s < d\gamma. \\
      \end{cases}
    \end{align}
  \end{subequations}
  We notice that the conditions in \eqref{eq:gmlmc_iso_ft_work} for
  the optimal convergence rate $\Order{\tol^{-2}}$ do not depend on
  the number of directions in the underlying problem. Moreover, for
  the case where the variance convergence is slower than the work
  increase rate, the work complexity of the MIMC estimator with both
  types of index sets is better than the work complexity of MLMC
  whenever we work with multi-directional problems, i.e., $d>1$.

    It should be noted, however, that the MIMC results require {\em
    mixed} regularity (in the sense of \textbf{Assumptions 1,2}) of
  a certain order. On the other hand, the MLMC results require only {\em
    ordinary} regularity of the same order. {Moreover, MIMC
    based on full tensor index sets requires condition
    \eqref{eq:ft_assump_dominant} to be satisfied. In this isotropic
    case, this condition simplifies to the following inequality:
    $2w > d \min(s,\gamma)$. This is, in some cases, more restrictive
    than the similar condition of MLMC, which reads in such a case as
    $2w \geq \min(s,d\gamma)$, cf. \cite[Theorem 2.3]{tsgu13}}. On the
  other hand, MIMC with optimal TD index sets has only the much less
  restrictive, dimension-independent condition $2 w \geq s$.
  Moreover, using TD index sets, the rate of work
  complexity is, up to a logarithmic factor, independent of $d$. In
  other words, up to a logarithmic term, the rate of the computational
  complexity of MIMC with an optimal TD index set is equivalent to the
  computational complexity of MLMC when used with a single
  direction.

          \end{remark}

\begin{remark}[Lower mixed regularity]
  In some cases, we might have enough mixed regularity in the sense of
  \textbf{Assumptions 1-2}
 along some directions but not along others. For example, assume that, out of $d$
  directions, the first $\tilde d$ directions do not have mixed regularity among
  each other.
  Our MIMC estimator can still be applied by considering all first $\tilde d$
  directions as a single direction. This is done by using  the same discretization
  parameter, $\tilde \alpha$, for all $\tilde d$ directions and then finding the new
  rates, $\tilde \gamma$, $\tilde s$ and $\tilde w$, of the resulting
  direction.
  This can be thought of as combining MLMC in the first $\tilde d$ directions
  with MIMC in the rest of the directions and in the case $d=\tilde d$, i.e., the problem
  has no mixed regularity and MIMC reduces to standard MLMC.
  All results derived in the current work can still be applied to
  this new setting, which conceptually corresponds to  $d-\tilde d+1$ directions in the MIMC
  results presented here.
  In particular, if we assume that the first $\tilde d$ directions
  are isotropic with the same variance convergence rate, $s$, and work rate, $\gamma$,
  then the results in Theorems~\ref{thm:gmlmc_ft} and \ref{thm:gmlmc_td_opt} deteriorate in the
  sense that the conditions relating $s$ and $\gamma$ in
  \eqref{eq:gmlmc_asym_cases} for the grouped direction
  are replaced by the
  more stringent conditions relating $s$ and $\tilde d\gamma$.
\end{remark}

\begin{remark}[A unique worst direction]
  In Theorem \ref{thm:gmlmc_td_opt}, consider the special case when
  ${\mathfrak{z}=1}$, i.e., when the directions are dominated by a
  single ``worst'' direction with the maximum difference between the
  work rate and the rate of variance convergence. In this case, the
  value of $L$ becomes
  \[L = \frac{1}{\eta }\left( \log(\tol_B^{-1}) +
    \log({C_{\textnormal{Bias}}}) \right)\]
  and MIMC with a TD index set in \textbf{Case B} achieves a better
  rate for the computational complexity, namely
  $\Order{\tol^{2-2\zeta}}$. In other words, the logarithmic term
  disappears in the computational complexity and, in this case, the
  computational complexity of MIMC with an optimal TD index set is, up
  to a constant, the same computational complexity as a MLMC along the
  single worst direction.

  The same results also hold when the variance convergence is faster
  than algebraic in all but one direction and, in this case, the
  overall complexity is dictated by the only direction with an
  algebraic convergence rate. In this case, the optimal index set
  might no longer be of TD-type but it can still be constructed by
  using the same methodology and profit definition presented in
  Section~\ref{ss:td_set}.
\end{remark}

\begin{remark}[Optimal weights and case of smooth noise]
  As Lemma~\ref{lem:gmlmc_td_work_gen} shows, even if the rates $s_i$
  and $w_i$ are not known for all $i \in \indset$ and if we choose
  arbitrary weights, $\vdelta$, to build the TD index set, we still
  obtain a work complexity whose rate is independent of the number of directions,
  $d$, up to a logarithmic term. The complexity is determined by the
  direction with the slowest weak convergence and the direction with
  the largest difference between the rate of variance convergence and
  the rate of work per sample. Moreover, recall the definition of the
  optimal $\vdelta$ in \eqref{eq:optimal_weights} and note that when
  $\xi=0$, $s_i = 2 w_i$ for all $i \in \indset$. In this case,
  $\delta_i = \frac{\log(\beta_i) \gamma_i}{2 C_\vdelta}$ and the
  optimal index set is completely determined by the rates in the work per
  sample along each direction.
\end{remark}
\inchangesA{
\begin{remark}[Rate of memory usage of MIMC]
\label{rem:memory}
Assume that the memory usage to calculate a sample of $\estS_\alpha$
is $\Order{\exp(\tau |\alpha|)}$ for some $\tau >0$.  In MIMC, when
using TD-type index sets in isotropic problems, we have
$|\alpha| \leq L$ and as such the maximum memory usage of MIMC in this
case is $\Order{\exp(\tau L)}$ or
$\Order{\tol^{\frac{-\tau}{w}} (\log\left(\tol^{-1}\right))^{\tau
    (d-1)}}$
where $w = w_i$ for all $i=1,\ldots,d$. Notice that the rate with
respect to $\tol$ is, up to a logarithmic term, dimension
independent. Compare this to MLMC, where the memory usage is
$\Order{\exp(d \tau \ell)}$ for $\ell \leq L$.  The maximum memory
usage of MLMC is hence $\Order{\exp(d \tau L)}$ or
$\Order{\tol^{\frac{-d\tau}{w}}}$. Refer to
Figure~\ref{fig:dof_vs_tol} for an illustration of this point.
\end{remark}
}

\makeatletter{}\section{Numerical Example}\label{s:res}
This section presents a numerical example illustrating the behavior of
the MIMC, which is in agreement with our theoretical analysis. For the
sake of comparison, we show the results of applying three different
approximations to the same problem: MLMC as outlined in \cite{cst13},
MIMC with a full tensor index set as outlined in
Section~\ref{ss:full_tensor}, and MIMC with
a total degree index set as outlined in Section~\ref{ss:td_set}. We
begin by describing the numerical example. Then, we present the solvers
and algorithms and finish by giving the numerical results.

\subsection{Example overview} \label{sec:prob-overview}
The numerical example is adapted from \cite{haji_opt} and is based on Example~\ref{ex:spde_problem} in Section~\ref{sec:hier_intro}
with some particular choices that satisfy the assumptions therein and \textbf{Assumptions 1-3}.
 First, the domain is chosen to be ${\mathcal{D} = [0,1]^3}$ and the
\inchangesA{ forcing is $f(\vec x;\omega) = 1$.
Moreover, the diffusion coefficient is chosen to be a function of two random variables as follows:
\begin{equation*}
    a(\vec x; \omega) = 1 + \exp \Big(2 Y_1 \Phi_{121}(\vec x) + 2 Y_2 \Phi_{877}(\vec x)\Big).
\end{equation*}
Here, $Y_1$ and $Y_2$ are i.i.d. uniform
random variables in the range $[-1,1]$. We also take
\begin{align*}
  \Phi_{ijk}(\vec x) &= \phi_i(x_1)\phi_j(x_2)\phi_k(x_3),\\
   \text{and}\qquad  \phi_i(x) &=
\begin{cases}
    \cos\left( \frac{i}{2} \pi x \right) & i \text{ is even}, \\
    \sin\left( \frac{i+1}{2} \pi x \right) & i \text{ is odd},
\end{cases}
\end{align*}
Finally, the quantity of interest, $S$, is
\[S = 100 \left( 2 \pi \sigma^2 \right)^\frac{-3}{2} \int_\mathcal{D}
\exp \left( - \frac{ \| \vec x - \vec x_0 \|^2_2}{2 \sigma^2} \right)
u(\vec x) d\vec x, \]
and the selected parameters are $\sigma=0.16$ and
${\vec x_0 = \left[0.5,0.2,0.6\right]}$.  A reference solution can be
calculated to sufficient accuracy by using stochastic collocation
\cite{bnt2010} with a sufficiently accurate quadrature to produce the
reference value, $\E{S}$.  Using this method, the reference value
$1.3301$ is computed with an error estimate of $10^{-4}$.
}
\subsection{Solvers and Algorithms}
\subsubsection{Solving the underlying PDE problems}
To solve the underlying PDE problems, uniform meshes with a standard trilinear {finite element} basis are used to discretize
the weak form of the model problem. The number of elements in each dimension is a positive integer, $N_i,$ to
give a mesh size of $h_i = N_i^{-1}$ for all $i=1,2,3$.
Moreover, we use the same $\beta=2$ in all dimensions. In other words, given
 a multi-index $\valpha$, we use $N_i = 4 \cdot 2^{\alpha_i}$ in each
 dimension and the resulting problem is isotropic with $w_i = 2$ and
 $s_i=4$ for all $i=1,2,3$ (the same case as
 Remark~\ref{rem:iso_work}).
The linear solver MUMPS ~\cite{Amestoy2001,Amestoy2006} was used for solving the linear problem. For the mesh sizes of interest,
the running time of MUMPS varies from quadratic to linear in the total number of degrees of freedom (cf. Figure \ref{fig:time_vs_dof}).
As such, $\gamma_i$ in \eqref{eq:wl_model} is the same for all $i=1,2,3$ and ranges from $1$ to $2$.

\begin{figure}
  \centering
  \includegraphics[scale=0.6]{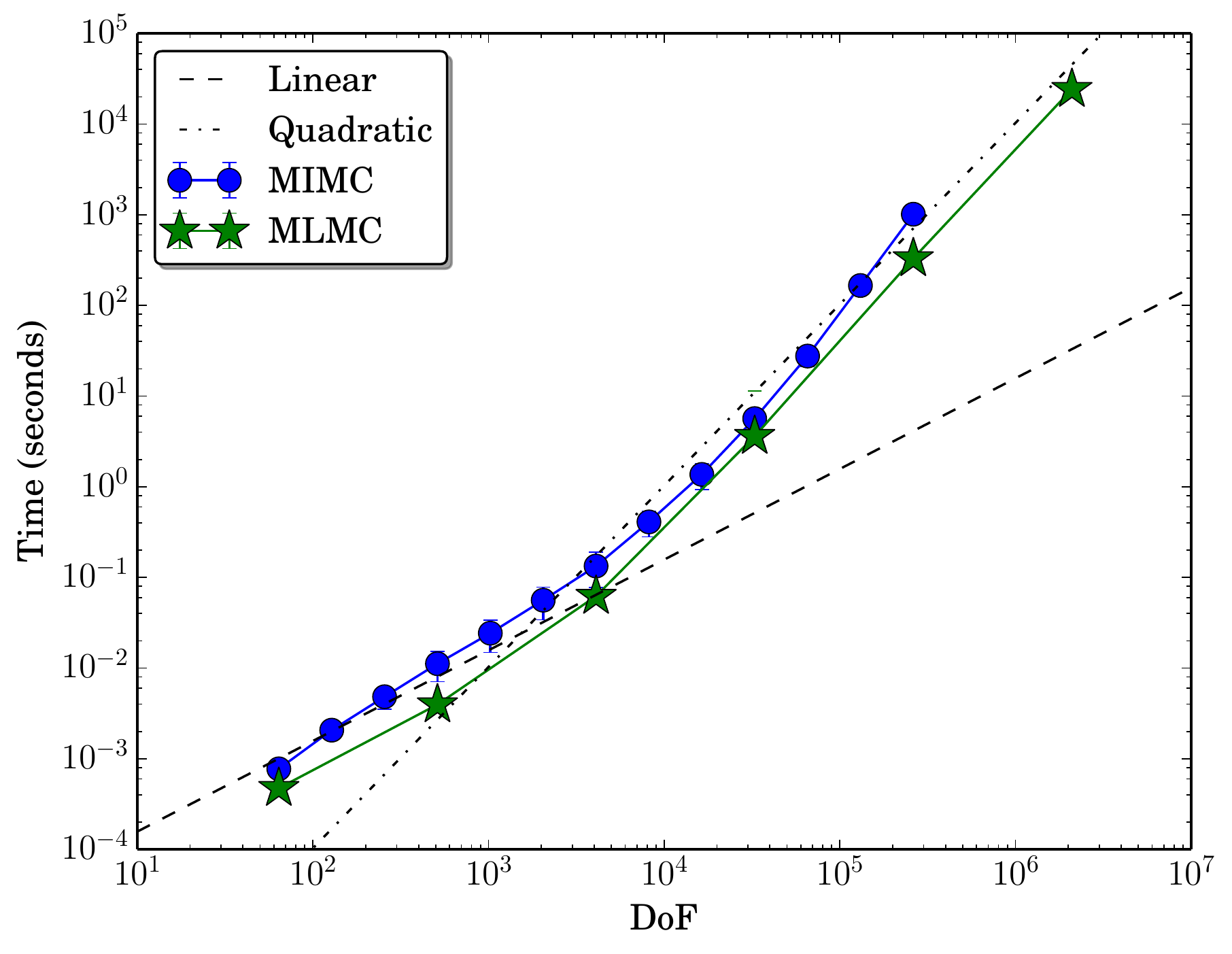}
  \caption{Average running time to estimate the difference operators
    for MIMC and MLMC versus the maximum number of degrees of freedom
    in those levels. \inchangesA{Notice that MIMC has a higher cost
    than MLMC has for the same maximum number of degrees of freedom
    (DoF). This is because $2^d = 8$ terms are estimated per
    difference level in MIMC compared with $2$ terms for difference
    levels in MLMC.  Moreover, because all dimensions are isotropic in
    our numerical example, this shows that $\gamma_i$ in
    \eqref{eq:wl_model} is the same for $i=1\ldots d$ and ranges from
    1 to 2.}}
  \label{fig:time_vs_dof}
\end{figure}

\subsubsection{MIMC Algorithm}
{The algorithm used to generate the results presented in the next section is a slight modification and extension of the MLMC algorithm first outlined in \cite{giles08}.}
Specifically, the sample variance was used to calculate the required
number of samples on each level in MIMC, with a minimum of $\overline{M_0}=5$ samples per level.
Moreover, we used fixed tolerance-splitting, $\theta=0.5$.
Note that this choice might be sub-optimal for MIMC and further work
needs to be done in this case.
The MIMC pseudo-algorithm can be summarized as follows
\begin{enumerate}[label=Step \arabic*.]
\item Set $k=1$
\item Ensure that at least $\overline{M_0}$ samples are calculated for all
  $\valpha \in \mathcal I_k$.
\item Using sample variances as estimates for $V_\valpha$, calculate
  $M_\valpha$ according to \eqref{eq:optimal_M} for all
  $\valpha \in \mathcal I_k$.
\item Calculate extra samples to have at least $M_\valpha$ samples for
  each ${\valpha \in \mathcal I_k}$.
\item Estimate the bias. We expand on this step below.
\item Stop if the bias estimate is less than $(1-\theta) \tol.$
\item Otherwise, increase $k$ and go to Step 2.
\end{enumerate}
This algorithm assumes that $\mathcal I_k \subset \mathcal I_{k-1}$.
For isotropic TD index sets, we simply
use \eqref{eq:idxset_anisoTD} with ${L=k/d}$.
For full tensor index sets, we increase the value of each $L_i$ for
$i=1,2,\ldots,d$ one at
a time cyclically.

\inchangesA{
In Step 6, we use the following bias estimate:
\begin{equation}
\label{eq:bias_comp_bound}
\text{Bias}(\mathcal I_k) \approx \left|\sum_{\alpha \in \partial
    \mathcal I_k}
\E{\estS_\alpha} \right|,
\end{equation}
where $\partial \mathcal I_k$ is the outer boundary of the index set,
$\mathcal I_k$.  Obviously, the error indicator
\eqref{eq:bias_comp_bound} does not provide an error bound in general
unless further assumptions on the integrand function $S$ are
made. Nevertheless, we use this error indicator in our problem
heuristically. We further approximate the expectations in
\eqref{eq:bias_comp_bound} by sample averages using the
$\max(\overline{M_0}, M_\alpha)$ samples that are available for every
$\valpha \in \partial \mathcal I_k$.

Finally, we apply the continuation concept from \cite{haji_CMLMC} by
running MIMC (and MLMC) with a sequence of larger tolerances than
$\tol$ to obtain increasingly accurate estimates of the sample
variances.
}

\subsection{Results}
Three methods were tested: MLMC as outlined in \cite{cst13}, MIMC with
full tensor index sets (referred to as ``FT'' in the figures), and
MIMC with isotropic total degree index sets (referred to as ``TD'' in
the figures). In this isotropic example, the total degree index sets
defined in Section \ref{ss:td_set} become
\[ \mathcal{I}(L) = \{\valpha\in \nset^3 : |\valpha| \leq 3L\}.\]
Recall that in this example, $d=3, s_i=s=4, w_i=w=2$ and
$\gamma_i = \gamma$ for all $i=1,2,3$, where $\gamma$ ranges from 1 to
2.  As such, the condition for MLMC, $2 w \geq \min(s, d\gamma)$, is
satisfied for $\gamma \in [1,2]$.  Similarly the condition for MIMC
with the optimal TD index set, $2w \geq s$, is satisfied. On
the other hand, the condition for MIMC with a full tensor index set,
$2w > d\min(s,\gamma)$, is {\em not} satisfied for
$\gamma \in \left[\frac{4}{3},2\right]$.
According to Remark~\ref{rem:iso_work}, for small enough tolerances
where $\gamma=2$ mostly, we expect the work complexity of MIMC with a TD
index set to be $\Order{\tol^{-2}}$. On the other hand, MIMC with a full
tensor index set would have a work complexity of
$\Order{\tol^{-\frac{d \gamma}{w}}} = \Order{\tol^{-3}}$.  Similarly, MLMC would have
a work complexity of $\Order{\tol^{-2 - \frac{d\gamma-s}{w}}} =
\Order{\tol^{-3}}$.

Figures \ref{fig:mimc_contours} and \ref{fig:mimc_pl_contours} show
that \textbf{Assumptions 1-3} are indeed satisfied (at least for
sufficiently fine discretizations). On the other hand,
Figures~\ref{fig:mlmc_el_vl} and \ref{fig:mimc_el_vl} show numerical
results that are in agreement with the convergence rates claimed
above.  Specifically, these figures show results consistent with the
values $s=4$ and $w=2$. Moreover, Figure~\ref{fig:qqplot} shows
numerical evidence of the normality of the statistical error of the
MIMC estimator. Figure \ref{fig:time_vs_tol} shows the running time
for different tolerances. The MIMC method with total degree index sets
seems to exhibit the expected rate of $\tol^{-2}$ in the computational
time. On the other hand, MLMC and MIMC with a full tensor index set
seem to exhibit a rate closer to $\tol^{-3}$, especially for smaller
tolerances.  The staircase-like effect in running time of MLMC and
MIMC with a full tensor index set is due to the discrete increments of
the maximum number of degrees of freedom per level (cf.
Figure~\ref{fig:dof_vs_tol}). Since a fixed tolerance-splitting
parameter, $\theta=0.5$, was used, the statistical constraint is not
relaxed when the bias becomes smaller and the algorithm ends up
solving for a slightly smaller tolerance than the required $\tol$
(cf. Figure~\ref{fig:error_vs_tol}). Notice that although the fixed
tolerance splitting parameter was also used for MIMC with total degree
index sets, the running time does not exhibit the same jumps. This is
because the discrete increments in the number of degrees of freedom
are not as significant in this method (cf.
Figure~\ref{fig:dof_vs_tol}).
\inchangesA{Finally, Figure~\ref{fig:dof_vs_tol}
can also be used to estimate the memory requirements of MIMC versus
MLMC.  In this figure we can see that using MIMC with total degree
index sets, allows us to achieve the same value of $\tol$ with
substantially fewer degrees of freedom. In fact, we were not able to
run MLMC or MIMC with full tensor index sets for very small tolerances
due to their memory requirements.

For comparison, Figure~\ref{fig:time_vs_tol_d4} shows the
running time of MLMC and MIMC with a TD index set  when applied to a similar
problem but in four dimensions instead of three. Here, we expect MLMC to have
a work complexity of $\Order{\tol^{-4}}$ while the expected rate of
MIMC with a TD index set is still  $\Order{\tol^{-2}}$. The tolerances
shown for MLMC were the only ones that we were able to compute with 64
gigabytes of memory.
}
\begin{figure}
  \centering
  \includegraphics[scale=0.37]{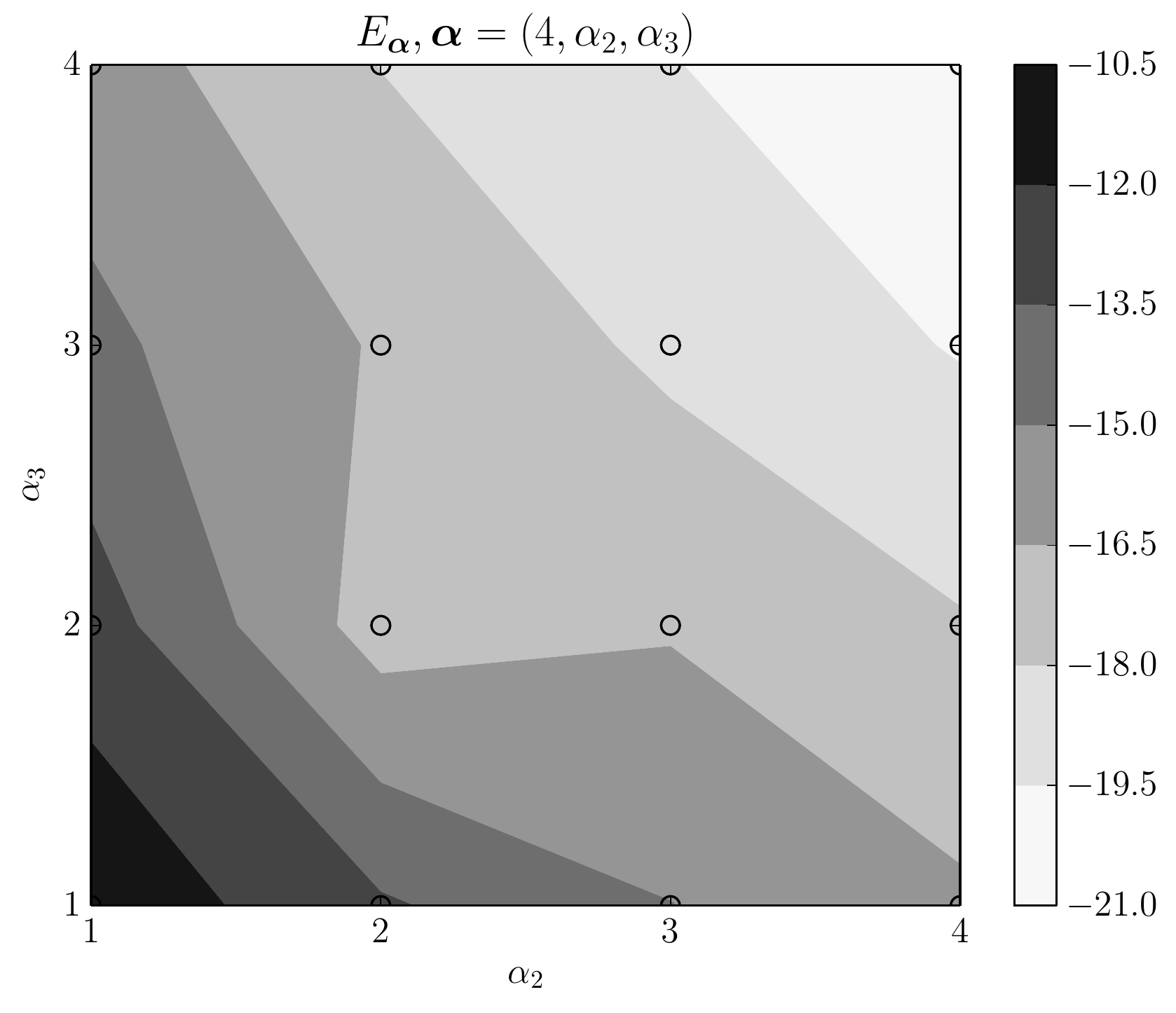}
  \includegraphics[scale=0.37]{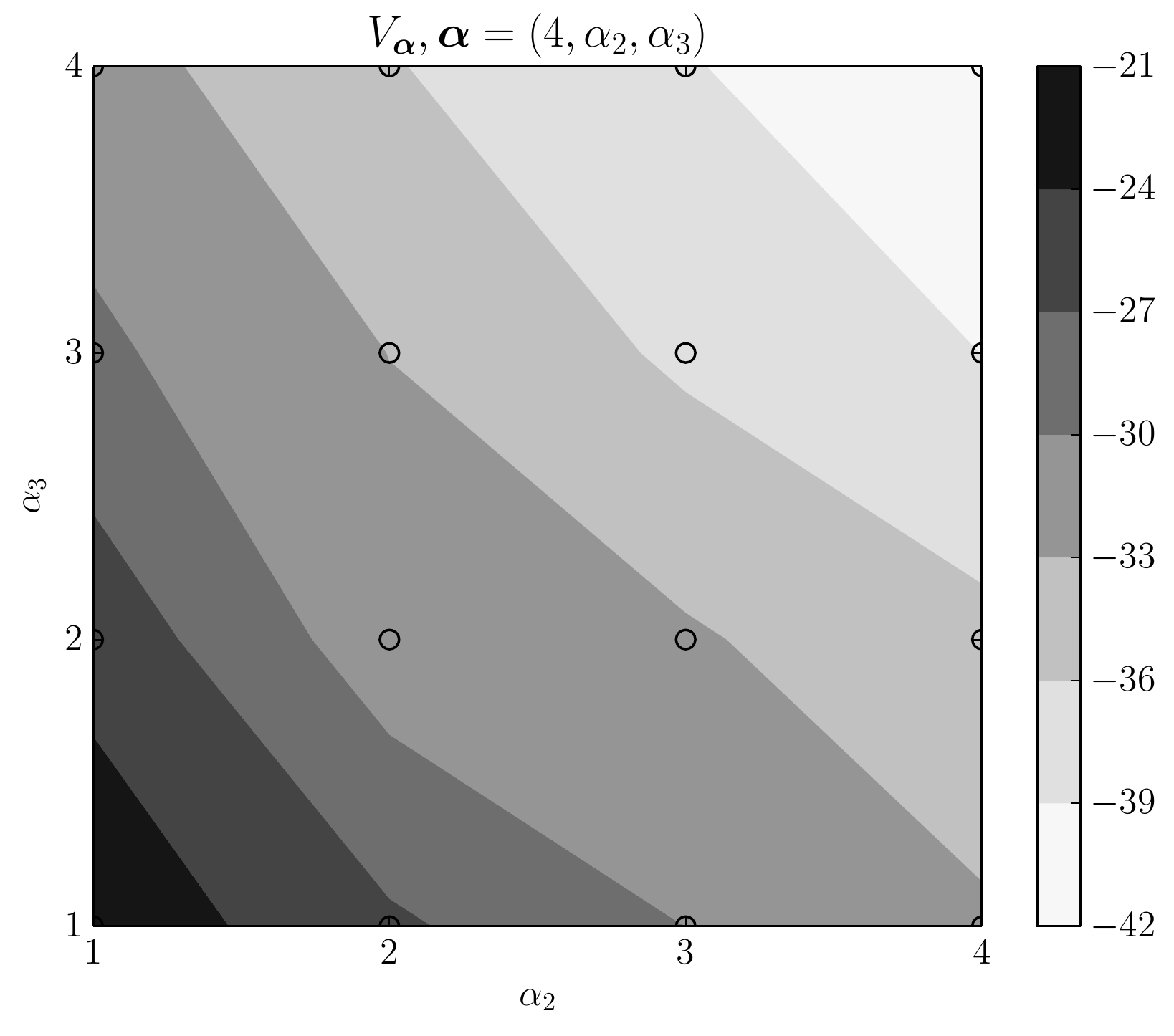}
  \caption{Numerical example, rate verification: contour plots of
    sample mean (left) and variance (right) of mixed differences used
    in MIMC for a slice of multi indices.}
  \label{fig:mimc_contours}
\end{figure}
\captionsetup{width=6cm}

\begin{figure}
  \begin{floatrow}

    \ffigbox[\FBwidth] {\caption{Numerical example, rate verification:
        contour plots of profits used in MIMC for a slice of multi
        indices.  The parallel lines, asymptotically, suggest that
        isotropic TD index sets are nearly optimal in this example.}
      \label{fig:mimc_pl_contours}}
    {\includegraphics[scale=0.37]{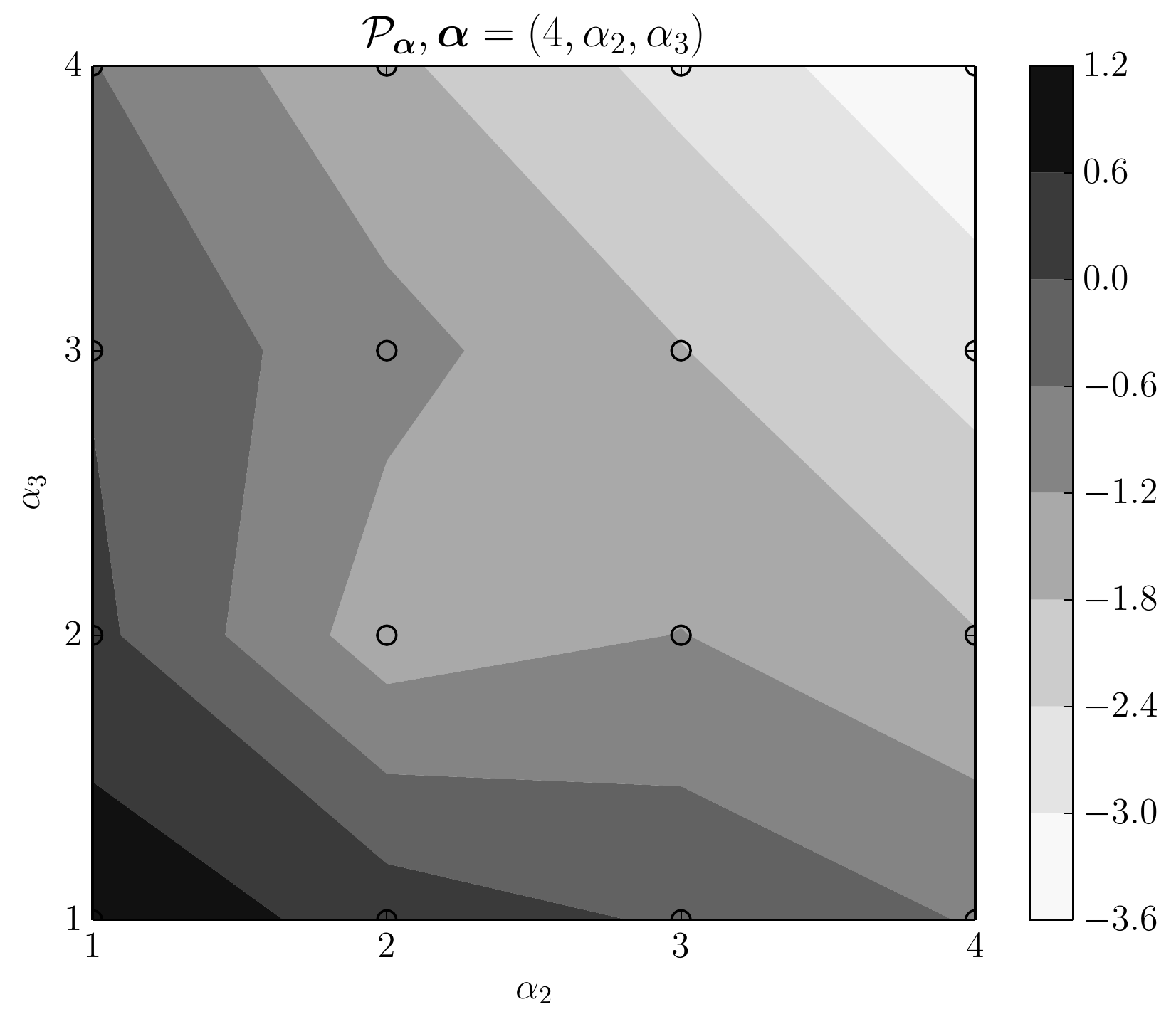}} \ffigbox[\FBwidth]
    {\caption{QQ-plot of the normalized error of the MIMC estimator
        with a TD index set for different values of $\tol$. Similar
        results were obtained for other tolerances using either MLMC
        or MIMC with full tensor index sets. This is in agreement with
        Lemma \ref{thm:clt_result}.}
      \label{fig:qqplot}} {\includegraphics[scale=0.33]{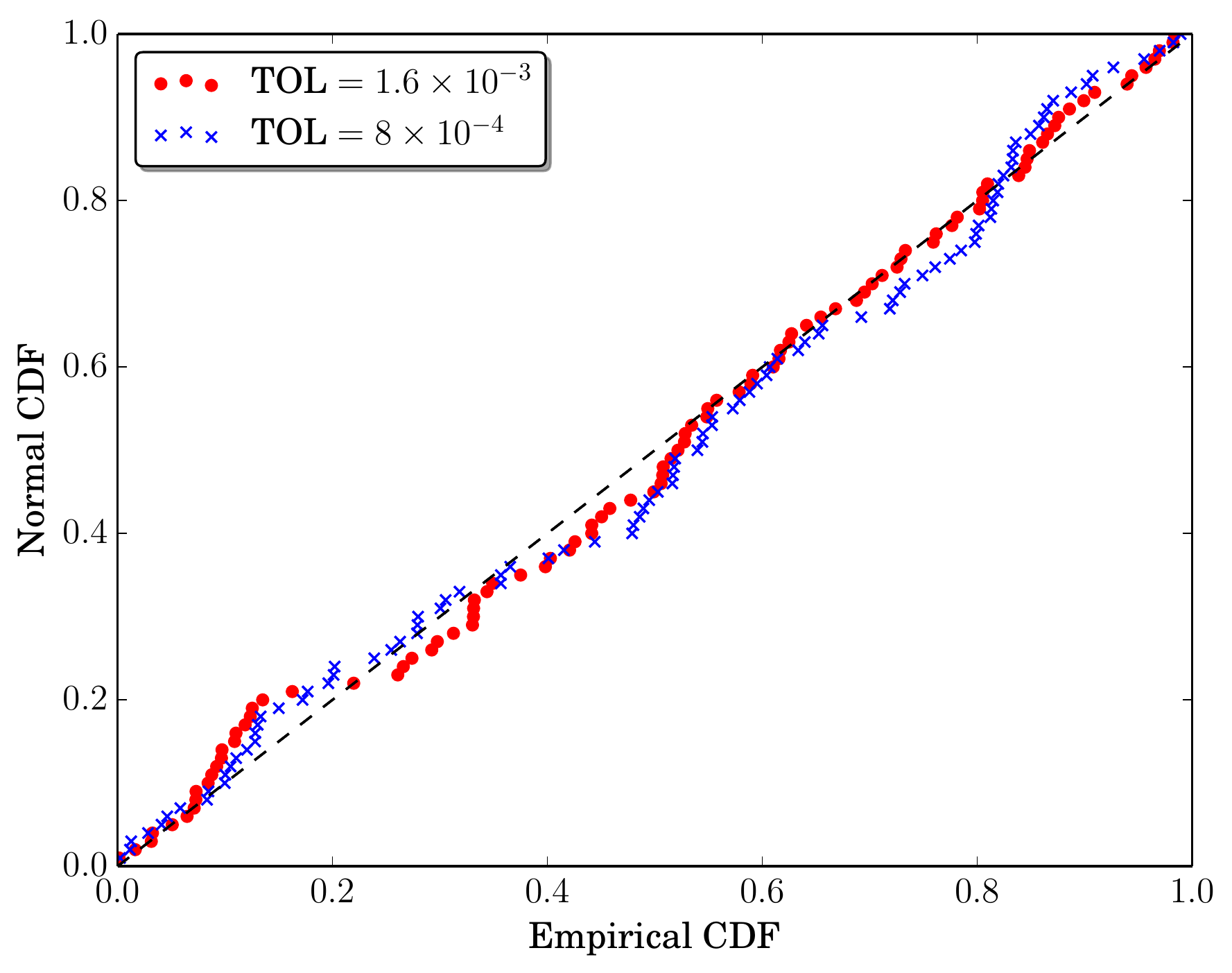}}
  \end{floatrow}
\end{figure}

\captionsetup{width=0.9\textwidth}

\begin{figure}
  \centering
  \includegraphics[scale=0.41]{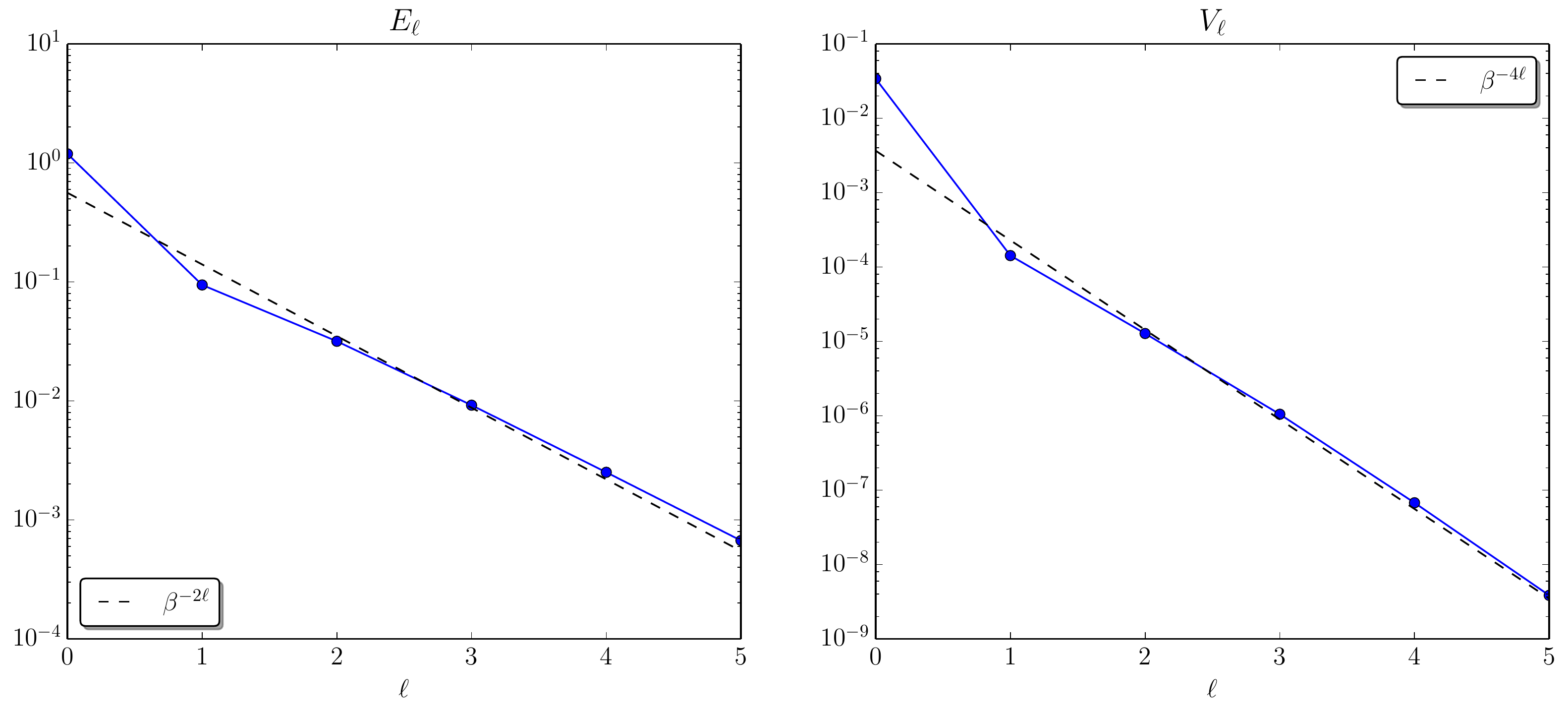}
  \caption{Numerical example, rate verification: sample mean (left) and variance (right) of differences versus level $\ell$ for MLMC.
  Notice that the observed rates are consistent with Remark~\ref{rem:assumptions}.}
  \label{fig:mlmc_el_vl}
\end{figure}

\begin{figure}
  \centering
  \includegraphics[scale=0.41]{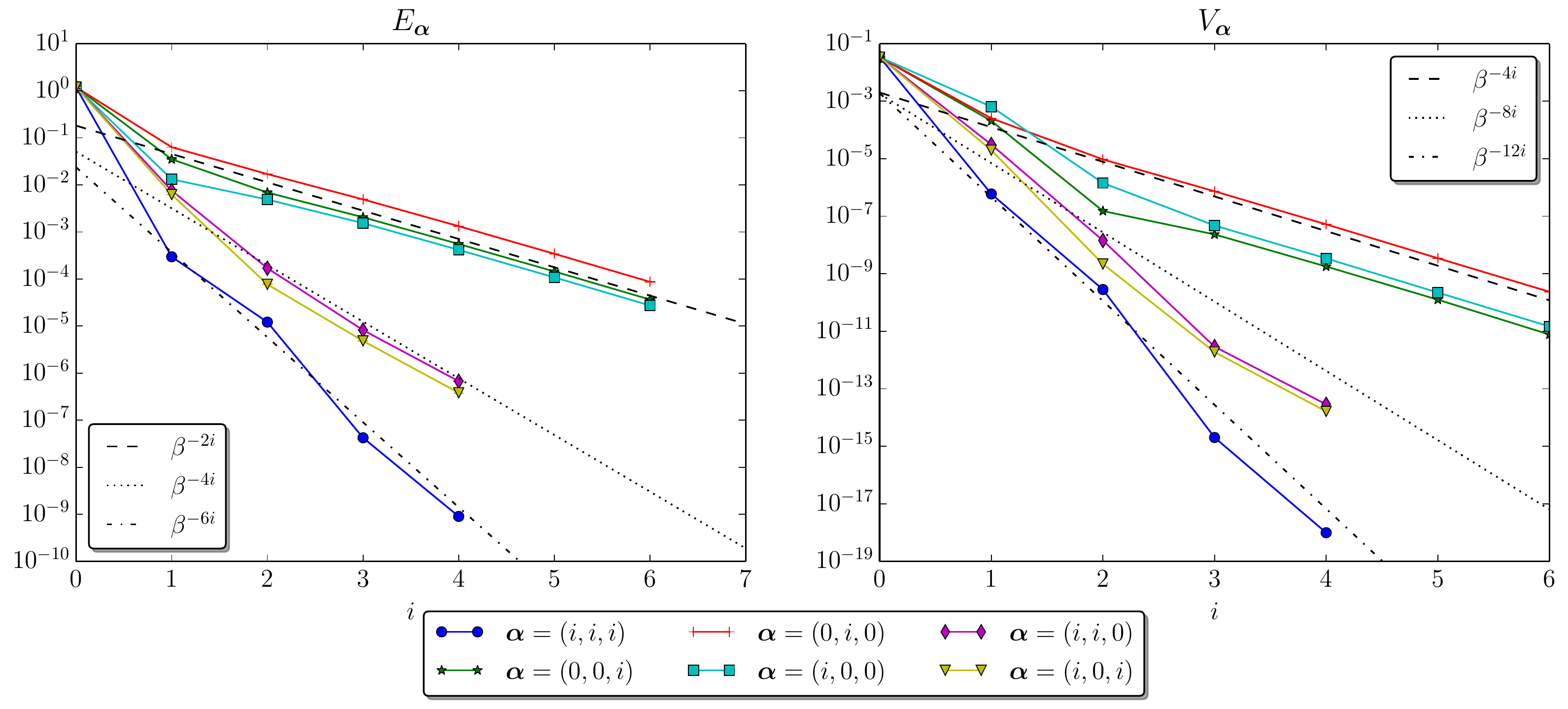}
  \caption{Numerical example, rate verification: sample mean (left) and variance (right) of mixed differences used in MIMC.
    Notice that the observed rates are consistent with Remark~\ref{rem:assumptions} and are
  better than those observed for MLMC, cf. Figure~\ref{fig:mlmc_el_vl}.}
  \label{fig:mimc_el_vl}
\end{figure}

\begin{figure}
  \centering
  \includegraphics[scale=0.6]{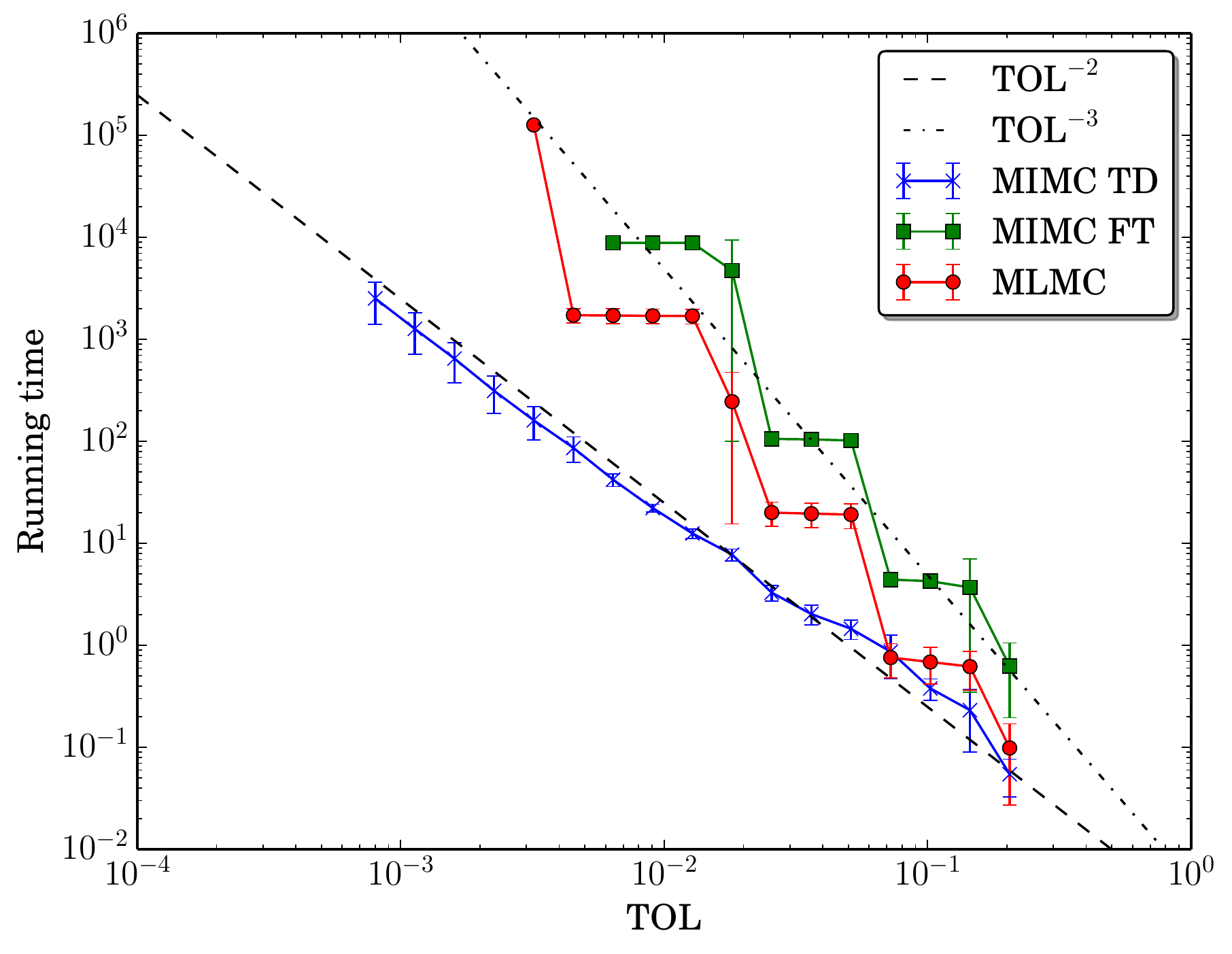}
  \caption{\inchangesA{Running time for different values of $\tol$
      when using MLMC and MIMC with different index sets. The error
      bars extend from the 5\% percentile to the 95\%
      percentile. Notice that the rate of MIMC is the optimal Monte
      Carlo rate of $\Order{\tol^{-2}}$ for this example, while MLMC
      is closer to $\Order{\tol^{-3}}$, in agreement with the results
      listed in Remark~\ref{rem:iso_work} for $d=3, \gamma=2, s=4$ and
      $w=2$. For comparison, recall that MC has a work complexity of
      $\Order{\tol^{-5}}$.}}
  \label{fig:time_vs_tol}
\end{figure}
\captionsetup{width=6cm}
\begin{figure}
  \begin{floatrow}
    \ffigbox[\FBwidth] { \caption{The exact computational error for
        MLMC and MIMC using different index sets. Notice that since we
        imposed a fixed tolerance splitting parameter, {$\theta
          =0.5$}, in some cases our numerical error is slightly
        smaller than the required $\tol$ for all methods.}
      \label{fig:error_vs_tol}
    } {\includegraphics[scale=0.33]{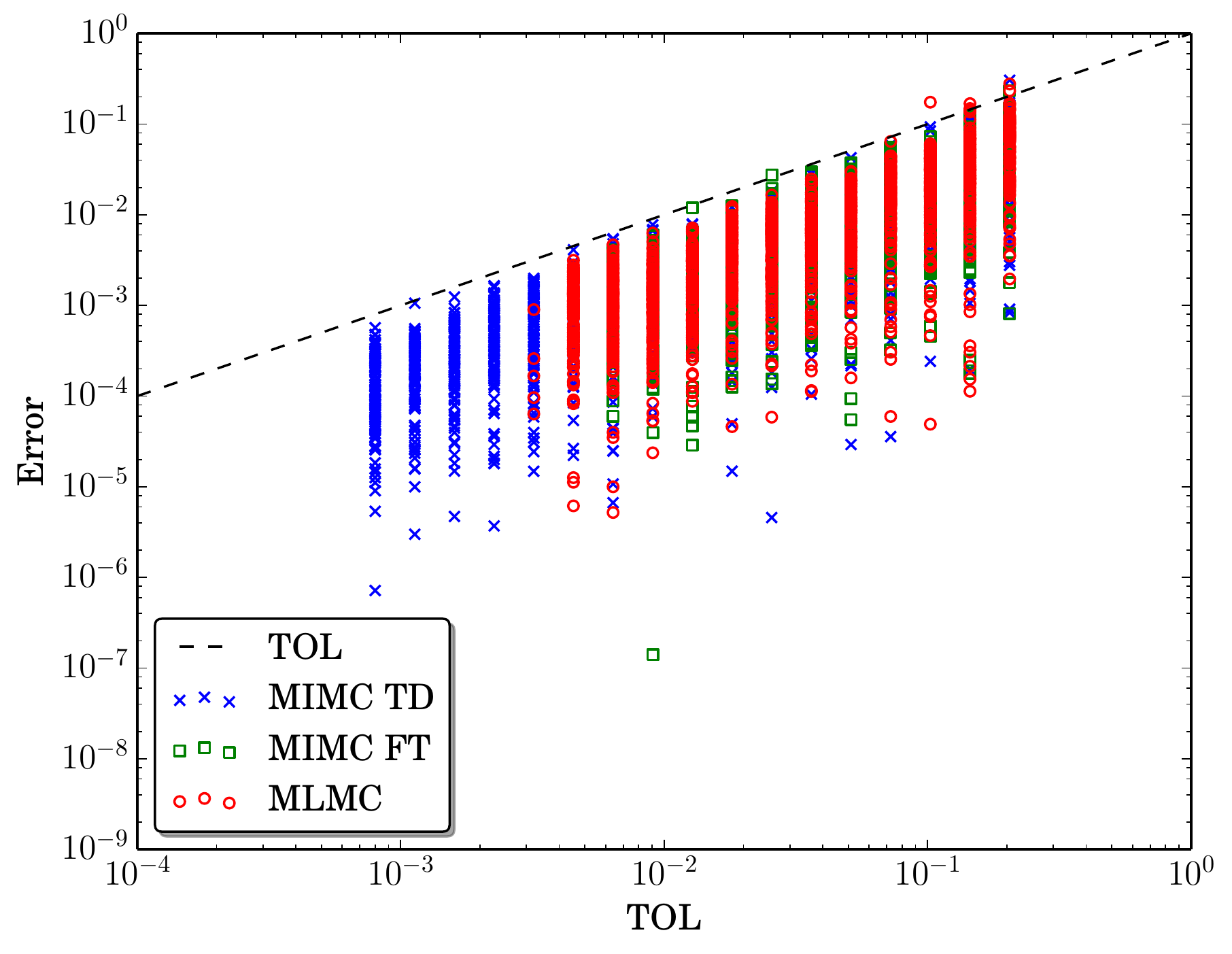}} \ffigbox[\FBwidth]
    { \caption{The maximum number of degrees of freedom across levels
        for MLMC and MIMC with different index sets. Notice that by
        using total degree index sets, we are able to achieve the same value
        of $\tol$ with substantially fewer degrees of freedom. Refer
        to Remark~\ref{rem:memory} for discussions regarding this point.}
      \label{fig:dof_vs_tol}
    }
    {\includegraphics[scale=0.33]{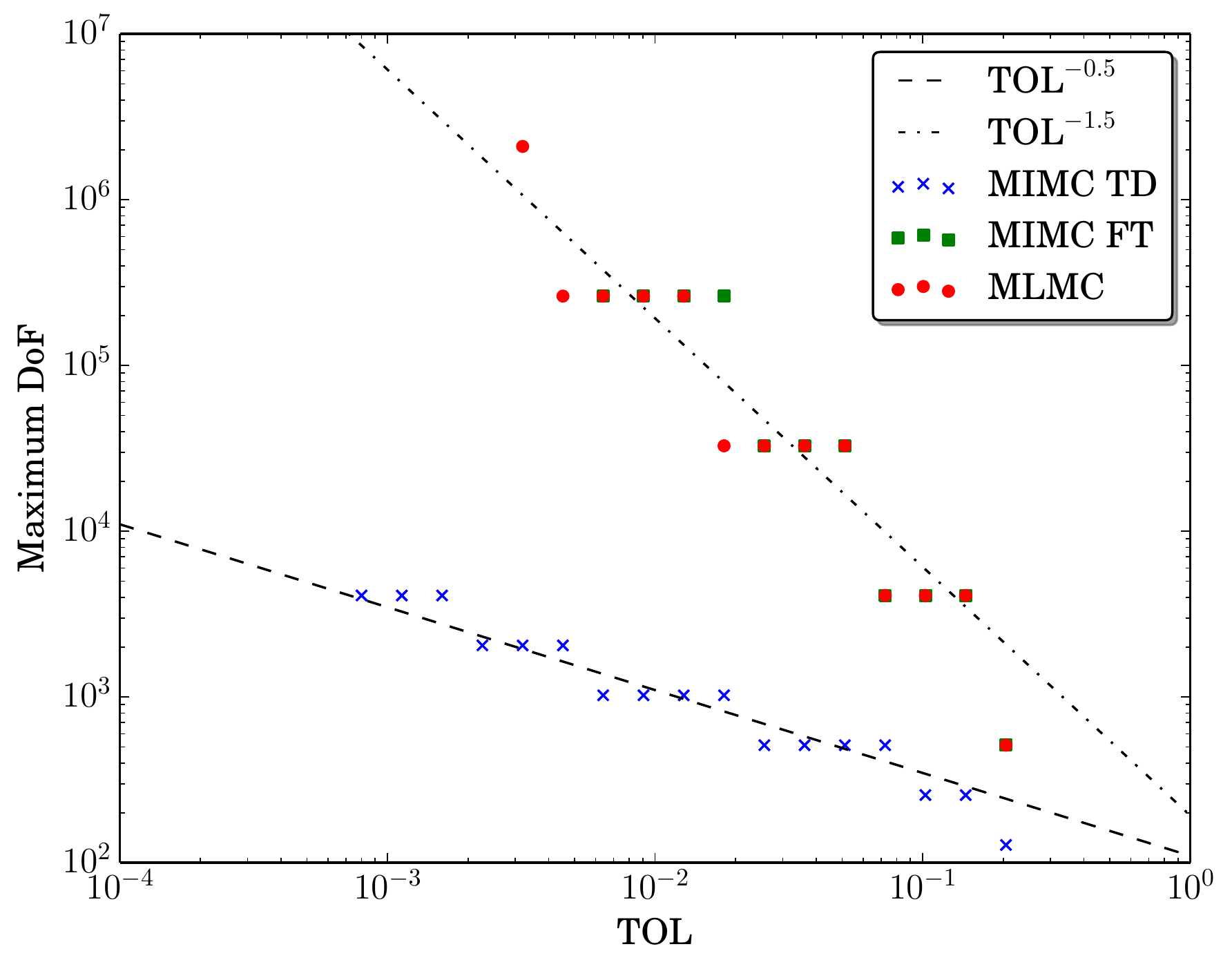}}
  \end{floatrow}
\end{figure}
\captionsetup{width=0.9\textwidth}
\begin{figure}
  \centering
  \includegraphics[scale=0.6]{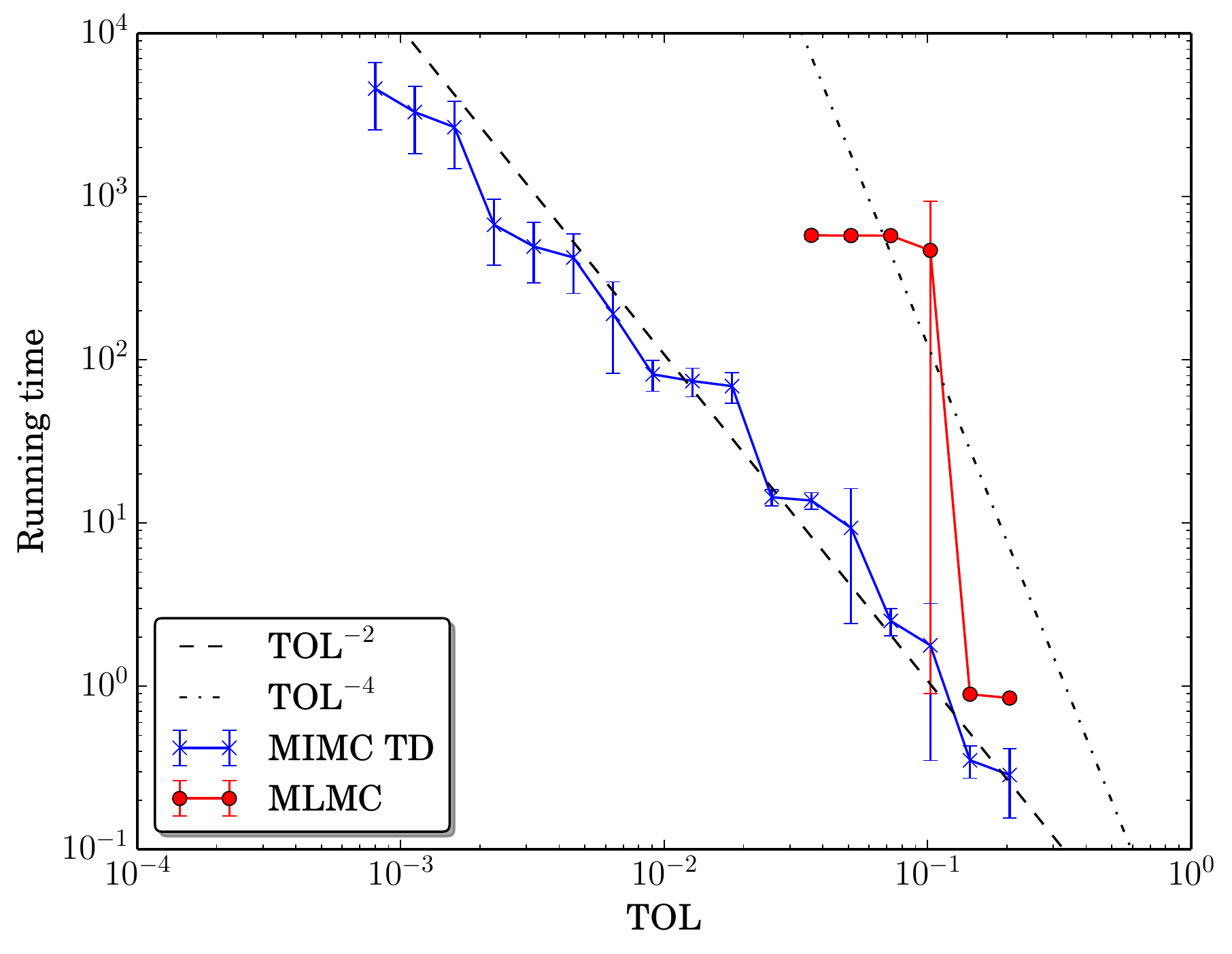}
  \caption{\inchangesA{Running time for different values of $\tol$ for a 4D
    problem when using MLMC and MIMC with TD index sets. The error
    bars extend from the 5\% percentile to the 95\% percentile.
    Notice that the work complexity of MIMC for this example has the
    optimal Monte Carlo rate of $\Order{\tol^{-2}}$.  According to
    Remark~\ref{rem:iso_work} for $d=3, \gamma=2, s=4$ and $w=2$,
    MLMC is expected to have a work complexity of
    $\Order{\tol^{-4}}$. However, we were not able to compute for
    smaller tolerances with 64 gigabytes of available memory. For
    comparison, recall that MC has a work complexity of $\Order{\tol^{-6}}$. }}
  \label{fig:time_vs_tol_d4}
\end{figure}

\makeatletter{} \section{Conclusions}\label{s:conc} We have proposed and analyzed a
 novel Multi-Index Monte Carlo (MIMC) method for weak approximation of
 stochastic models that are described in terms of differential
 equations either driven by random measures or with random
 coefficients. The MIMC method uses a stochastic combination technique
 to solve the given approximation problem, generalizing the notion of
 standard MLMC levels into a set of multi indices that should be
 properly chosen to exploit the available regularity. Indeed, instead
 of using first-order differences as in standard MLMC, MIMC uses
 high-order differences to reduce the variance of the hierarchical
 differences dramatically. This in turn gives a new improved
 complexity result that increases the domain of the problem parameters
 for which the method achieves the optimal convergence rate,
 $\mathcal{O}(\tol^{-2}).$ We have outlined a method for constructing
 an optimal index set of indices for our MIMC method. Moreover, under
 our standard assumptions, we showed that the optimal index set turns
 out to be the total degree (TD) type. Using optimal index sets, MIMC
 achieves a better rate for the computational complexity than when
 using full tensor index sets; in fact, the rate does not depend on
 the dimensionality of the underlying problem, up to logarithmic
 factors.  \inchangesA{Similarly, the rate of required memory for MIMC
   with respect to $\tol$ is, up to a logarithmic terms,
   dimension-independent (unlike MLMC) allowing us to solve for
   smaller tolerances than is possible with MLMC.}  {In addition, for
   MIMC with TD index sets, the conditions on the weak convergence
   rate for achieving such rates are dimension-independent and less
   stringent, compared with similar conditions for MLMC and MIMC with
   full tensor index sets.} We also presented numerical results to
 substantiate some of the derived computational complexity rates. In
 Appendix~\ref{app:clt}, using the Lindeberg-Feller theorem, we also
 show the asymptotic normality of the statistical error in the MIMC
 estimator and justify in this way our error estimate that allows both
 the required accuracy and confidence level in the final result to be
 prescribed.

Our method requires more regularity of the underlying solution than
does MLMC. If the underlying solution is sufficiently regular only
in some directions, then one can still combine MIMC with MLMC by applying mixed first-order differences to the sufficiently regular
directions, while applying a single first-order difference to less regular directions.

In future work, more has to be done to improve the MIMC algorithm,
using the variance convergence
model to estimate the variances instead of relying
on sample variance only; for example, by applying ideas such as those in \cite{haji_CMLMC}.
{Also}, a better choice of the splitting parameter, $\theta$, can be derived to improve the computational complexity up to a constant factor; similar to
the work done in \cite{haji_opt}.
Moreover, MIMC can be used to improve the computational complexity rate in the case of PDEs with random  fields that are approximated by converging series, such as a Karhunen-Lo\'eve
decomposition, cf. \cite{tsgu13}.
By treating the number of terms in the decomposition as an extra discretization direction and applying MIMC, we might be able to improve the computational complexity. Also, the use of either \emph{a priori} refined non-uniform discretizations or adaptive algorithms based on \emph{a posteriori} error estimates for non-uniform refinement as introduced in \cite{hsst12,hsst13,Moraes2014_MLMC_HTL} can be combined with MIMC to improve efficiency.
Finally, ideas from \cite{van2014multilevel} {and \cite{Harbrecht_H13_MLSC}} can be extended by replacing the Monte Carlo sampling of mixed differences in MIMC by a sparse-grid stochastic collocation,
effectively including interpolation levels along the different random directions into the combination technique
together with the other discretization parameters. Similarly, we can apply Quasi Monte Carlo to replace Monte Carlo sampling of the mixed differences in MIMC as outlined in \cite{kss12_MLQMC} for a multilevel setting. Provided that there is enough mixed regularity in the problem at hand, we expect to improve again the optimal complexity further from $\mathcal{O}(\tol^{-2})$ in MIMC to $\mathcal{O}(\tol^{-r})$ with $r<2$.

\section*{Acknowledgments} Ra\'{u}l~Tempone is a member of the Special
Research Initiative on Uncertainty
Quantification (SRI-UQ), Division of Computer, Electrical and Mathematical Sciences and Engineering (CEMSE)
    at King Abdullah University of Science and Technology (KAUST).
The authors would like to recognize the support of KAUST AEA project ``Predictability and Uncertainty
Quantification for Models of Porous Media'' and University of Texas at Austin AEA Round 3
``Uncertainty quantification for predictive modeling of the dissolution of porous and fractured media''.
The second author acknowledges the support of the Swiss National Science Foundation under the Project No.
140574 ``Efficient numerical methods for flow and transport phenomena in heterogeneous random porous media''.
The authors would also like to thank Prof. Mike Giles for his valuable comments
on this work.
\appendix
\makeatletter{}\section{Asymptotic Normality of the MIMC estimator}
\label{app:clt}
\begin{lemma}[Asymptotic Normality of the MIMC Estimator]
\label{thm:clt_result}
Consider the MIMC estimator introduced in \eqref{eq:gmlmc}, $\mathcal{A}$, based on a set of multi indices, $\mathcal{I}(\tol)$, and given by
  \begin{equation*} \mathcal{A} = \sum_{\valpha \in \mathcal{I}}
    \sum_{m=1}^{M_\valpha} \frac{\estS_{\valpha}(\omega_{\valpha, m})}{M_\valpha}.
    \end{equation*}
  Assume that for $1\le i\le d$ there exists  $0<L_i(\tol)$ such that
\begin{equation}\label{eq:set_cond}
 \mathcal{I}(\tol) \subset \{\valpha\in \nset^d : \alpha_i \leq L_i(\tol), \text{ for } 1\le i\le d\}.
\end{equation}
      Denote ${Y_\valpha = |\estS_\valpha - \E{\estS_\valpha}|}$ and assume that the following inequalities
    \begin{subequations}
    \begin{align}
        \label{eq:2nd-moment-bound} Q_S \prod_{i=1}^d \exp(-\alpha_i s_i) &\leq \E{Y_{\valpha}^2}, & \\
        \label{eq:4th-moment-bound} \E{Y_{\valpha}^{2+\rho}}    &\leq Q_R \prod_{i=1}^d \exp(-\alpha_i r_i),&
      \end{align}
        \end{subequations}
hold  for strictly positive constants $\rho, \{s_i, r_i\}_{i=1}^d, Q_S$ and $Q_R$.
  Choose the number of samples on each level, $M_\valpha(\tol)$, to satisfy,
  for  strictly positive sequences  $\{\tilde s_i\}_{i=1}^d$ and $\{H_\vtau\}_{\vtau \in \mathcal{I}(\tol)}$
  and for all $\valpha \in \mathcal{I}(\tol)$,
  \begin{align}
    M_\valpha \geq \tol^{-2} \,C_M \left(\prod_{i=1}^d \exp(-\alpha_i\tilde s_i) \right)  H_\valpha^{-1} \left( \sum_{\vtau \in \mathcal{I}(\tol)} H_\vtau \right).
 \label{eq:chosen_ml}
\end{align}
Denote, for all $1 \leq i \leq d,$
\begin{equation}
  \label{eq:clt_cond_pi}
  p_i = (\rho/2) \tilde s_i  - r_i + (1+\rho/2) s_i
\end{equation}
and choose $0 < c_i$ such that whenever $0 < p_i$,  the inequality $c_i < \rho/p_i$ holds.
Finally, if we take the quantities $L_i(\tol)$ in \eqref{eq:set_cond} to be
  $$
 L_i(\tol) = c_i \log(\tol^{-1}) + \order{\log(\tol^{-1})}, \text{ for all }1 \leq i \leq d,
 $$
then we have
\[\lim_{\tol \downarrow 0} \prob{\frac{\mathcal{A}- \E{\mathcal A}}{\sqrt{\var{\mathcal A}}} \leq z} = \Phi \left(z \right), \]
  where $\Phi (z)$ is the normal cumulative distribution function of a standard normal random variable.
\end{lemma}
\begin{proof}
  We prove this theorem by ensuring that the Lindeberg condition
  \cite[Lindeberg-Feller Theorem, p. 114]{Durret1996} (also restated in \cite[Theorem A.1]{haji_CMLMC}) is satisfied.
    The condition becomes in this case
    \[ \lim_{\tol \downarrow 0} \underbrace{\frac{1}{\var{\mathcal A}} \sum_{\valpha \in \mathcal{I}(\tol)} \sum_{m=1}^{M_\valpha} \E{ \frac{Y_\valpha^2}{M_\valpha^2} \mathbf{1}_{\frac{Y_\valpha}{M_\valpha}
        > \epsilon \sqrt{\var{\mathcal A}}}}}_{= F} = 0,\]
    for all $\epsilon > 0$. Below we make repeated use of the following identity for non-negative sequences ${\{a_\valpha \}}$ and ${\{b_\valpha \}}$ and $q \geq 0$:
    \begin{equation} \sum_{\valpha}{a_\valpha^q b_\valpha} \leq \left( \sum_\valpha a_\valpha \right)^q \sum_\valpha b_\valpha. \label{eq:identity} \end{equation}
    First, we use the Markov inequality to bound
    \begin{align*}
        F &=  \frac{1}{\var{\mathcal A}} \sum_{\valpha \in \mathcal{I}(\tol)} \sum_{m=1}^{M_\valpha} \E{ \frac{Y_\valpha^2}{M_\valpha^2} \mathbf{1}_{Y_\valpha >
                    \epsilon \sqrt{\var{\mathcal A}} M_\valpha}} \\
        &\leq \frac{\epsilon^{-\rho}}{\var{\mathcal A}^{1+\rho/2}}  \sum_{\valpha \in \mathcal{I}(\tol)}  M_\valpha^{-1-\rho} \E{Y_\valpha^{2+\rho}}.
    \end{align*}
    Using \eqref{eq:identity} and substituting for the variance
    $\var{\mathcal A}$ where we denote $\var{\estS_\valpha} = \E{\left( \Delta\mathcal  S_\valpha - \E{ \Delta
          \mathcal S_\valpha} \right)^{2}}$ by $V_\valpha$, we find
    \begin{align*}
        F &\leq \frac{\epsilon^{-\rho} \left( \sum_{\valpha \in \mathcal{I}(\tol)}  M_\valpha^{-1} V_\valpha \right)^{1+\rho/2}}{\left( \sum_{\valpha \in \mathcal{I}(\tol)} V_\valpha M_\valpha^{-1} \right)^{1+\rho/2}}
                \sum_{\valpha \in \mathcal{I}(\tol)} V_\valpha^{-1-\rho/2} M_\valpha^{-\rho/2} \E{Y_\valpha^{2+\rho}} \\
        &= \epsilon^{-\rho}
                \sum_{\valpha \in \mathcal{I}(\tol)} V_\valpha^{-1-{\rho}/{2}} M_\valpha^{-{\rho}/{2}} \E{Y_\valpha^{2+\rho}}.
    \end{align*}
     Using the lower bound      in {\eqref{eq:chosen_ml}} on the number of samples, $M_\valpha$,
          and \eqref{eq:identity}, again yields
    \begin{align*}
        F &\leq C_M^{-\rho/2} \epsilon^{-\rho} \tol^\rho \left(
          \sum_{\valpha \in \mathcal{I}(\tol)}
          V_\valpha^{-1-{\rho}/{2}} \left(\prod_{i=1}^d
            \exp\left(\frac{\rho \alpha_i\tilde s_i}{2}\right) \right)
          H_\valpha^{\rho/2} \E{Y_\valpha^{2+\rho}} \right) \\ & \hskip8em
                \left( \sum_{\vtau \in \mathcal{I}(\tol)} H_\vtau \right)^{-\rho/2} \\
        &\leq C_M^{-\rho/2} \epsilon^{-\rho} \tol^\rho \left( \sum_{\valpha \in \mathcal{I}(\tol)} V_\valpha^{-1-{\rho}/{2}}\left(\prod_{i=1}^d \exp\left(\frac{\rho \alpha_i\tilde s_i}{2}\right) \right) \E{Y_\valpha^{2+\rho}} \right).
                \end{align*}
    Finally, using the bounds \eqref{eq:2nd-moment-bound} and \eqref{eq:4th-moment-bound},
    \begin{align*}
      F &\leq \underbrace{C_M^{-\rho/2} \epsilon^{-\rho} Q_S^{-1-{\rho}/{2}} Q_R}_{= C_F}
     \tol^\rho  \left(
 \sum_{\valpha \in \mathcal{I}(\tol)}
          \left(\prod_{i=1}^d \exp\left(p_i \alpha_i \right) \right)
        \right).
                      \end{align*}
       Next, define three sets of dimension indices:
\begin{equation*}
\aligned
\hat I_1 &= \{1\le i \le d: p_i < 0\}, \\
\hat I_2 &= \{1\le i \le d: p_i = 0\}, \\
\hat I_3 &= \{1\le i \le d: p_i > 0\}.
\endaligned
\end{equation*}
Then, using \eqref{eq:set_cond} yields
      \[
      \aligned
      F \leq& C_F \tol^\rho \prod_{i=1}^d \left(  \sum_{\alpha_i=0}^{L_i}  \exp\left(p_i \alpha_i \right) \right)\\
      \leq &  C_F \tol^\rho \prod_{i \in \hat I_1} \frac{1}{1-\exp(p_i)} \prod_{i \in \hat I_2} L_i \prod_{i \in \hat I_3}
      \frac{1 - \exp(p_i (L_i+1))}{1-\exp(p_i)}.
      \endaligned
      \]
            To conclude, observe that if $|\hat I_3| = 0$, then $\lim_{\tol \downarrow 0} F= 0$ for any choice of $L_i\geq  0$, $1\le i\le d$.
      Similarly, if $|\hat I_3| > 0$, since we assumed that $c_i p_i < \rho$ holds for all $i \in \hat I_3$, then $\lim_{\tol \downarrow 0} F= 0$.
\end{proof}

\begin{remark*}
  The lower bound on the number of samples per index \eqref{eq:chosen_ml} mirrors choice \eqref{eq:optimal_M},
  the latter being the optimal number of samples satisfying  constraint \eqref{eq:stat_const}.
  Specifically, $H_\valpha = \sqrt{V_\valpha W_\valpha}$ and $\tilde s_i = s_i$.
  Furthermore, notice that the previous Lemma bounds the growth of $L$ from above,
  while Theorem~\ref{thm:gmlmc_ft} and Theorem~\ref{thm:gmlmc_td_opt} bound the value of $L$ from below to satisfy the bias accuracy constraint.
\end{remark*}

\section{Integrating an exponential over a
  simplex}\label{app:exp_simplex}
\begin{lemma}
  \label{lem:exp_simplex}
The following identity holds for any $L > 0$ and $a \in \rset$:
\begin{equation}\label{eq:int_identity}
\aligned
\intlim{\vec x \in \rsetp^d}{|\vec x| \leq L} \exp(a |\vec x|)
\textnormal{d}\vec x &= (-a)^{-d} \left( 1 - \exp(La) \sum_{j=0}^{d-1}
  \frac{(-La)^j}{j!} \right)\\
&= \frac{1}{(d-1)!} \int_{0}^{L} \exp(at) t^{d-1}\:\textnormal{d}t.
\endaligned
\end{equation}
\end{lemma}
\begin{proof}
  \[
\aligned
\intlim{\vec x \in \rsetp^d}{|\vec x| \leq L} \exp(a |\vec x|)
\textnormal{d}\vec x &= L^d \intlim{\vec x \in \rsetp^d}{|\vec x| \leq 1}
\exp(aL |\vec x|)
\textnormal{d}\vec x.
\endaligned
\]
Then, we prove, by induction on $d$ and  for $b = aL$, the following identity:
\[ \intlim{\vec x \in \rsetp^d}{|\vec x| \leq 1} \exp(b |\vec x|)
\textnormal{d}{\vec x} =(-b)^{-d} \left( 1 - \exp(b) \sum_{j=0}^{d-1}
  \frac{(-b)^j}{j!} \right).
   \]
First, for $d=1$, we have
\[ \int_0^1 \exp(b x) \textnormal{d}x = \frac{\exp(b) - 1}{b}.
 \]
Next, assuming that the identity is true for $d-1$, we prove it for $d$. Indeed, we have
 \begin{align*}
&\intlim{\vec x \in \rsetp^d}{|\vec x| \leq 1} \exp(b |\vec x|)
\textnormal{d}\vec x
\\& = \int_0^1 \exp(b y) \left( \intlim{\vec x \in \rsetp^{d-1}}{|\vec
     x| \leq 1-y} \exp(b |\vec x|)
   \textnormal{d}\vec x  \right) \textnormal{d}y \\
&= \int_0^1 \exp(b y) \left( 1-y \right)^{d-1} \left( \intlim{\vec x \in \rsetp^{d-1}}{|\vec
     x| \leq 1} \exp((1-y) b |\vec x|)
   \textnormal{d}\vec x  \right) \textnormal{d}y \\
 &= \int_0^1 \exp(b y) \frac{\left( 1-y \right)^{d-1}}{(-(1-y)b)^{d-1}} \left( 1 - \exp((1-y)b) \sum_{j=0}^{d-2}
   \frac{(-(1-y)b)^j}{j!} \right) \textnormal{d}y \\
  &= \int_0^1 \left[ \frac{\exp(b y)}{(-b)^{d-1}}  - \frac{\exp(b)}{(-b)^{d-1}} \sum_{j=0}^{d-2}
    \frac{(-(1-y)b)^j}{j!}  \right] \textnormal{d}y \\
  &= \frac{\left( -1 \right)^{d-1}}{b^d} \left( \exp(b) -1 \right)
  - \frac{(-1)^{d-1} \exp(b)}{b^{d-1}} \sum_{j=0}^{d-2}
  \frac{(-b)^j}{(j+1)!} \\
  &= \frac{\left( -1 \right)^{d}}{b^d} -
  \frac{\left( -1 \right)^{d}}{b^d} \exp(b)
  - \frac{(-1)^{d} \exp(b)}{b^{d}} \sum_{j=1}^{d-1}
  \frac{(-b)^j}{(j)!} \\
    &= (-b)^{-d} \left( 1 - \exp(b) \sum_{j=0}^{d-1}
  \frac{(-b)^j}{j!} \right).
\end{align*}
Finally, the second equality in \eqref{eq:int_identity} follows by repeatedly integrating by parts.
\end{proof}
\inchangesA{
\begin{lemma}\label{lem:work_bound}
For $a \in \rset^d$, assume $A = \max_{i=1,2\ldots d} a_i > 0$ and denote
\begin{align*}
  \mathfrak{a}_1 &= \#\left\{i = 1,2,\ldots d \::\: a_i =
    A\right\},\qquad
 & \mathfrak{a}_2 &=d-\mathfrak{a}_1.
\end{align*}
Then, for any $L > 0$, there exists an $\epsilon > 0$ satisfying
\[
\epsilon \leq A - \max \left( 0, \max_{\substack{i=1,2\ldots d \\ a_i
      < A}} a_i \right),
\]
such that the following inequality holds:
\begin{equation}
  \aligned
  \intlim{\vec x \in \rsetp^d}{|\vec x| \leq L} \exp(\vec a \cdot
  \vec x)\textnormal{d}\vec x
  &\leq \mathfrak{C_W}(\vec a) \exp\left(A L\right) L^{\mathfrak{a}_1 -1}.
\endaligned
\end{equation}
Here, the constant $\mathfrak {C_W}(\vec a)$ is given by
\begin{equation}
  \label{eq:exp_int_bound_CW}
\mathfrak {C_W}(\vec a) =
\begin{cases}
  \frac{1}{A(d-1)!} & \text{ if $\mathfrak a_1 = d$} \\
  \frac{4}{\epsilon(2A - \epsilon)} \frac{\exp(1-\mathfrak a_2
    )}{(\mathfrak a_1-1)! (\mathfrak a_2-1)!}  \left( \frac{2
      (\mathfrak a_2 -1)}{\epsilon} \right)^{\mathfrak{a}_2 -1} &
  \text{otherwise}
\end{cases}
\end{equation}
\end{lemma}
\begin{proof}
  First, note that $\vec a = A \vec 1$ for some scalar $A > 0$ and $\vec
  1 = (1,1,\ldots, 1) $ if and
  only if $\mathfrak{a}_1 = d$. Then Lemma \ref{lem:exp_simplex}
  immediately gives
  \[
  \intlim{\vec x \in \rsetp^{d} }{|\vec x| \leq L} \exp(
     \vec a \cdot \vec x) \textnormal{d}\vec x
     \leq \frac{L^{d-1} \exp(A L)}{A (d-1)!}.
     \]
     Otherwise, recall that
  \begin{equation}
    \label{eq:exp_bound}
    x^j \leq \left(\frac{j}{b}\right)^j \exp(-j) \exp(bx)
  \end{equation}
  holds for any $x>0, b>0$ and $j\in \nset$. Then, using
  Lemma \ref{lem:exp_simplex} and \eqref{eq:exp_bound}, we can write
   \begin{align*}
     &\intlim{\vec x \in \rsetp^{d} }{|\vec x| \leq L} \exp(
     {\vec a} \cdot \vec x) \textnormal{d}\vec x \\
     &\leq \intlim{\vec x_2 \in \rsetp^{\mathfrak{a}_2 } }{|\vec x_2|
       \leq L} \exp\left(\left(A - \epsilon\right) |\vec x_2|\right) \left(\intlim{\vec x_1
         \in \rsetp^{\mathfrak{a}_1 } }{|\vec x_1| \leq L - |\vec
         x_2|} \exp(A |\vec x_1|)
       \textnormal{d}\vec x_1 \right) \textnormal{d}\vec x_2 \\
     &= \frac{1}{(\mathfrak{a}_1-1) !} \intlim{\vec x_2 \in
       \rsetp^{\mathfrak{a}_2 } }{|\vec x_2| \leq L} \exp\left(
       (A - \epsilon) |\vec x_2|\right) \left(\int_0^{L - |\vec x_2|} \exp(A t)
       t^{\mathfrak{a}_1 -1} \textnormal{d}t \right)
     \textnormal{d}\vec x_2 \\
     &= \frac{1}{(\mathfrak{a}_1-1) !} \int_0^{L} \exp(A t)
     t^{\mathfrak{a}_1 -1} \left( \intlim{\vec x_2 \in
         \rsetp^{\mathfrak{a}_2 } }{|\vec x_2| \leq L -t}
       \exp\left((A - \epsilon) |\vec x_2|\right) \textnormal{d}\vec x_2 \right)
     \textnormal{d}t
     \\
     &= \frac{1}{(\mathfrak{a}_1-1) !(\mathfrak{a}_2 -1)!} \int_0^{L}
     \exp(A t) t^{\mathfrak{a}_1 -1} \left( \int_0^{L-t}
       \exp\left(( A - \epsilon) z\right)z^{\mathfrak{a}_2 -1} \textnormal{d}z \right)
     \textnormal{d}t\\
     &\leq {\mathfrak{C}} \int_0^{L} \exp(A t)
     t^{\mathfrak{a}_1 -1} \left( \int_0^{L-t} \exp\left(
         \frac{z(2A - \epsilon) }{2} \right) \textnormal{d}z \right)
     \textnormal{d}t,
   \end{align*}
\[ \text{where}\qquad {\mathfrak{C}} = \frac{\exp(1-\mathfrak{a}_2 )}{(\mathfrak{a}_1-1)! (\mathfrak{a}_2-1)!}
 \left( \frac{2 (\mathfrak{a}_2 -1)}{\epsilon}
\right)^{\mathfrak{a}_2 -1},\]
continuing
\begin{align*}
  & \intlim{\vec x \in \rsetp^{d} }{|\vec x| \leq L} \exp(
  {\vec a} \cdot \vec x) \textnormal{d}\vec x \\
  &\leq {\mathfrak{C}} \exp\left(\frac{L(2 A - \epsilon) }{2}
     \right) \frac{2}{2A - \epsilon} \int_0^{L}
  t^{\mathfrak{a}_1 -1}
  \exp\left(\frac{\epsilon t}{2} \right) \textnormal{d}t \\
  &\leq {\mathfrak{C}} \exp\left(\frac{L(2 A - \epsilon) }{2} \right) \frac{2 L^{\mathfrak{a}_1
      -1}}{2 A -\epsilon} \int_0^{L}
  \exp\left(\frac{\epsilon t}{2} \right) \textnormal{d}t \\
  &\leq \frac{4{\mathfrak{C}}}{\epsilon (2A - \epsilon)} \exp\left(A L\right) L^{\mathfrak{a}_1
    -1}.
\end{align*}
\end{proof}
}

\begin{lemma} \label{lem:bias_bound}
The following inequality holds for any $L \geq 1$ and $\vec a \in \rsetp^d$:
\begin{equation*}
  \aligned
  \intlim{\vec x \in \rsetp^d}{|\vec x| > L} \exp(-\vec a \cdot
  \vec x)\textnormal{d}\vec x
  &\leq \mathfrak{C_B}(\vec a) \exp\left(-A L\right) L^{\mathfrak{a}_1 -1},
\endaligned
\end{equation*}
where
\begin{equation}
\label{eq:exp_int_bound_CB}
\mathfrak{C_B}(\vec a) =
\begin{cases}
  \sum_{j=0}^{d-1} \frac{A^{j-d}}{j!} & \text{ if $\mathfrak a_1 =
    d$} \\
  {\left( A + \epsilon \right)}^{-\mathfrak{a}_2}
  \sum_{j=0}^{\mathfrak{a}_1-1} \frac{A^{j-\mathfrak{a}_1}}{j!} +
  \frac{ 2\sum_{j=0}^{\mathfrak{a}_2-1} \exp(-j) \left( \frac{2
        j}{\epsilon} \right)^{j}
    \frac{{\left( A + \epsilon \right)}^{j-\mathfrak{a}_2}}{j!}}{
    (\mathfrak{a}_1-1)! \epsilon } &
  \text{otherwise}
\end{cases}
\end{equation}
and
\begin{align*}
  A &= \min_{i=1,2\ldots d} a_i, \qquad
  \epsilon = \min_{\substack{i=1,2\ldots d \\ a_i >
      A}} a_i  - A, \\
  \mathfrak{a}_1 &= \#\left\{i = 1,2,\ldots d \::\: a_i =
    A\right\}. \\
  \mathfrak{a}_2 &=d-\mathfrak{a}_1.
\end{align*}
\end{lemma}

\begin{proof}
  First, note that $\vec a = A \vec 1$ for some scalar $A > 0$ and $\vec 1 =
  (1,1,\ldots, 1)$ if and only if $\mathfrak{a}_1 = d$. Then Lemma
  \ref{lem:exp_simplex} immediately gives:
  \[
  \aligned &\intlim{\vec x \in \rsetp^{d} }{|\vec x| > L} \exp(- A
  |\vec x|) \textnormal{d}\vec x \\
  &= \int_{\vec x \in \rsetp^{d} } \exp( -A |\vec x|) \textnormal{d}\vec x -
  \intlim{\vec x \in \rsetp^{d} }{|\vec x| \leq L}
  \exp( -A |\vec x|) \textnormal{d}\vec x \\
  &= A^{-d} - A^{-d} \left( 1 - \exp(-AL) \sum_{j=0}^{d-1}
    \frac{(AL)^j}{j!} \right) \\
  &\leq \exp(-AL) L^{d-1} \sum_{j=0}^{d-1}  \frac{A^{j-d}}{j!}  \\
     \endaligned
     \]
     Otherwise, without loss of generality, assume that $a_i \leq a_j$
     for all $1\leq i \leq j\leq d$. Then, again using Lemma
     \ref{lem:exp_simplex}, we can write
\begin{align*}
  &\intlim{\vec x \in \rsetp^d}{|\vec x| \geq L}
  \exp\left(-\sum_{i=1}^d a_i x_i \right)\:
  \textnormal{d}\vec{x} \\
  &\leq \intlim{\vec x \in \rsetp^d}{|\vec x| \geq L}
  \exp\left(-A  \sum_{i=1}^{\mathfrak{a}_1} x_i -{\left( A + \epsilon \right)} \sum_{i=\mathfrak{a}_1+1}^d x_i\right)\: \textnormal{d}\vec{x} \\
  &= \Bigg[
  \int_{\vec x \in \rsetp^d} \exp\left(- A
    \sum_{i=1}^{\mathfrak{a}_1} x_i -{\left( A + \epsilon \right)}
    \sum_{i=\mathfrak{a}_1+1}^d x_i\right) \textnormal{d}\vec{x} - \\
  &\hskip6em \intlim{\vec x \in \rsetp^d}{|\vec x| \leq L}
  \exp\left(-A \sum_{i=1}^{\mathfrak{a}_1} x_i -{\left( A + \epsilon \right)}
    \sum_{i=\mathfrak{a}_1+1}^d x_i \right)\: \textnormal{d}\vec{x} \Bigg],
\end{align*}
where
\[\int_{\vec x \in \rsetp^d} \exp\left(- A \sum_{i=1}^{\mathfrak{a}_1} x_i
  -(A + \epsilon) \sum_{i=\mathfrak{a}_1+1}^d x_i\right) \textnormal{d}\vec{x} = A ^{-\mathfrak{a}_1} {\left( A + \epsilon \right)}^{-\mathfrak{a}_2}.\]
Now consider
\begin{align*}
 &\intlim{\vec x \in \rsetp^d}{|\vec x| \leq L} \exp\left(-
    A \sum_{i=1}^{\mathfrak{a}_1} x_i -{\left( A + \epsilon \right)} \sum_{i=\mathfrak{a}_1+1}^d x_i
  \right)\: \textnormal{d}\vec{x}
  \\
  &=\intlim{\vec x_2 \in \rsetp^{\mathfrak{a}_2}}{|\vec x_2| \leq L} \exp\left(-
    {\left( A + \epsilon \right)} |\vec x_2| \right) \\
  &\hskip6em \left(  \intlim{\vec x_1 \in
    \rsetp^{\mathfrak{a}_1}}{|\vec x_1| \leq L - |\vec x_2|} \exp\left(- A
    |\vec x_1| \right)\: \textnormal{d}\vec{x_1}\right) \:
  \textnormal{d}\vec{x_2}\\
  &= \frac{1}{(\mathfrak{a}_1-1)!} \intlim{\vec x_2 \in \rsetp^{\mathfrak{a}_2}}{|\vec x_2|
    \leq L} \exp\left(- {\left( A + \epsilon \right)} |\vec x_2| \right) \left(
    \int_0^{L-|\vec x_2|} \exp(-A t) t^{\mathfrak{a}_1-1} \textnormal{d}t \right)
  \textnormal{d}\vec{x_2}\\
  &= \frac{1}{(\mathfrak{a}_1-1)!} \int_0^{L} \exp(-A t) t^{\mathfrak{a}_1-1}
  \left(\intlim{\vec x_2 \in \rsetp^{\mathfrak{a}_2}}{|\vec x_2| \leq L-t}
    \exp\left(- {\left( A + \epsilon \right)} |\vec x_2| \right) \textnormal{d}\vec{x_2}
  \right) \textnormal{d}t
  \\
  &= \frac{1}{(\mathfrak{a}_1-1)! (\mathfrak{a}_2-1)!} \int_0^{L} \exp(-A t) t^{\mathfrak{a}_1-1}
  \left(\int_0^{L-t} \exp(-{\left( A + \epsilon \right)} z) z^{\mathfrak{a}_2-1}\: \textnormal{d}z
  \right)
  \textnormal{d}t \\
  &= \frac{{\left( A + \epsilon \right)}^{-\mathfrak{a}_2}}{(\mathfrak{a}_1-1)!} \int_0^{L} \exp(-A
  t) t^{\mathfrak{a}_1-1} \left(1 - \exp(-{\left( A + \epsilon \right)} (L-t))
    \sum_{j=0}^{\mathfrak{a}_2-1} \frac{({\left( A + \epsilon \right)} (L-t))^j}{j!} \right)
  \textnormal{d}t \\
  &= A ^{-\mathfrak{a}_1} {\left( A + \epsilon \right)}^{-\mathfrak{a}_2} - A ^{-\mathfrak{a}_1} {\left( A + \epsilon \right)}^{-\mathfrak{a}_2}\left(\exp(-A L) \sum_{j=0}^{\mathfrak{a}_1-1} \frac{(A
      L)^j}{j!} \right) \\&\qquad - \frac{{\left( A + \epsilon \right)}^{-\mathfrak{a}_2}}{(\mathfrak{a}_1-1)!} \int_0^{L} \exp(-A t) t^{\mathfrak{a}_1-1}\left(
    \exp(-{\left( A + \epsilon \right)} (L-t)) \sum_{j=0}^{\mathfrak{a}_2-1}
    \frac{({\left( A + \epsilon \right)} (L-t))^j}{j!} \right)
  \textnormal{d}t.
\end{align*}
Here, we can bound
\[
 A ^{-\mathfrak{a}_1} {\left( A + \epsilon \right)}^{-\mathfrak{a}_2}\left(\exp(-A L) \sum_{j=0}^{\mathfrak{a}_1-1}
   \frac{(A L)^j}{j!}  \right) \leq
 A ^{-\mathfrak{a}_1} {\left( A + \epsilon \right)}^{-\mathfrak{a}_2}\exp(-A L) L^{\mathfrak{a}_1 -1} \sum_{j=0}^{\mathfrak{a}_1-1}
  \frac{A ^j}{j!}.
\]
Recall that ${\epsilon} > 0$ and bound, using
\eqref{eq:exp_bound} for $(L-t)^j$ with $b = {\epsilon}/{2}$,
\begin{align*}
  &\frac{{\left( A + \epsilon \right)}^{-\mathfrak{a}_2}}{(\mathfrak{a}_1-1)!} \int_0^{L} \exp(-A t)
  t^{\mathfrak{a}_1-1}\left( \exp(-{\left( A + \epsilon \right)} (L-t)) \sum_{j=0}^{\mathfrak{a}_2-1}
    \frac{({\left( A + \epsilon \right)} (L-t))^j}{j!} \right)
  \textnormal{d}t \\
            &\leq \frac{{\left( A + \epsilon \right)}^{-\mathfrak{a}_2}}{(\mathfrak{a}_1-1)!} \left(
    \sum_{j=0}^{\mathfrak{a}_2-1} \exp(-j) \left( \frac{2 j}{ \epsilon } \right)^{j} \frac{{\left( A + \epsilon \right)}^j}{j!} \right) \\ &\hskip6em
  \exp\left(-L\left(\frac{ 2 A + \epsilon  }{2} \right)\right) \int_0^{L}
  \exp\left(\frac{\epsilon t }{2} \right) t^{\mathfrak{a}_1-1}
  \textnormal{d}t \\
  &\leq \frac{{\left( A + \epsilon \right)}^{-\mathfrak{a}_2}}{(\mathfrak{a}_1-1)!} \left(
    \frac{2}{\epsilon} \right) \left(
    \sum_{j=0}^{\mathfrak{a}_2-1} \exp(-j) \left( \frac{2 j}{ \epsilon
      } \right)^{j} \frac{{\left( A + \epsilon \right)}^j}{j!} \right)
  \exp\left(-A L\right)
  L^{\mathfrak{a}_1-1}.
\end{align*}
\end{proof}

\section{List of Definitions}
\label{app:defs}
In this section, for easier reference, we list definitions of notation that is used in multiple pages or sections
throughout the current work

\begin{tabular}{rl|}
\constRecall{$\tol_S$}{eq:tolS}
\constRecall{$\workEst$}{eq:gmlmc_worktilde_model}
\constRecall{$\workRem$}{eq:gmlmc_work1_model}
\constRecall{$\widetilde B$}{eq:gmlmc_biastilde_model}
\constRecall{$\tol_B$}{eq:tolB}
\constRecall{$\indset, \indsetA, \indsetB, \indsetC, \indsetBC$}{eq:ind_sets}
\constRecall{$d_1, d_2, d_3, \hat d$}{eq:ind_sets_sizes}
\constRecall{$\mathcal{C_B}$}{eq:fulltensor_C_Bias}
\constRecall{$\overline s_i, \overline w_i, \overline
  \gamma_i$}{eq:rates_not}
\constRecall{$\overline {\vec s}, \overline {\vec w}, \overline {\vec
    \gamma}$}{eq:rates_not_vec}
\end{tabular}
\begin{tabular}{rl}
\constRecall{$\chi, \eta, \gamma, \zeta, \xi$}{eq:max_rates}
\constRecall{$\mathfrak{x}, \mathfrak{e}, \mathfrak{g},
  \mathfrak{z}$}{eq:max_rates}
\constRecall{$C_{\textnormal{Bias}}$}{eq:CBias_td}
\constRecall{$C_A$}{eq:CW_caseA}
\constRecall{$C_B$}{eq:CW_caseB}
\constRecall{$C_R$}{eq:work_rem_constant}
\constRecall{$\mathscr{I}_C$}{eq:ind_caseC}
\constRecall{$\mathscr{I}_D$}{eq:ind_caseD}
\constRecall{$p_i$}{eq:clt_cond_pi}
\constRecall{$\mathfrak{C_W}$}{eq:exp_int_bound_CW}
\constRecall{$\mathfrak{C_B}$}{eq:exp_int_bound_CB}
\end{tabular}

\bibliographystyle{siam}

\end{document}